\documentclass[a4paper,11pt]{article}

\usepackage{soul}
\usepackage{comment}
\usepackage[a4paper,margin=2.8cm]{geometry}
\usepackage[utf8]{inputenc}
\usepackage{amstext}
\usepackage{amsfonts}
\usepackage{amsthm}
\usepackage{textcomp}
\usepackage{amssymb}
\usepackage{amscd}
\usepackage{amsmath}
\usepackage{xcolor}
\usepackage{tikz-cd}
\usepackage{caption}
\usepackage{enumitem}
\usepackage{theoremref}
\usepackage{mathtools}
\usepackage{float}
\usepackage{etoc}
\usepackage{stmaryrd}
\setlist[enumerate]{label=({\arabic*})}

\theoremstyle{theorem}
\newtheorem{thm}{Theorem}[section]

\newtheorem{prop}[thm]{Proposition}
\newtheorem{cor}[thm]{Corollary}

\theoremstyle{definition}
\newtheorem{defn}[thm]{Definition}
\newtheorem{exmp}[thm]{Example}
\newtheorem{rmk}[thm]{Remark}

\allowdisplaybreaks

\newcommand{\N}{\mathbb{N}}
\newcommand{\R}{\mathbb{R}}
\newcommand{\Z}{\mathbb{Z}}
\newcommand{\Q}{\mathbb{Q}}
\newcommand{\K}{\mathbb{K}}

\newcommand{\B}{{\mathcal B}}

\newcommand{\eps}{{\varepsilon}}
\newcommand{\esssup}{{\mathrm{ess\,sup}}}
\newcommand{\essinf}{{\mathrm{ess\,inf}}}

\newcommand{\Per}{\operatorname{Per}}
\newcommand{\PER}{\operatorname{PER}}
\newcommand{\card}{\operatorname{card}}
\newcommand{\ldens}{\operatorname{\underline{dens}}}

\DeclareRobustCommand{\rchi}{{\mathpalette\irchi\relax}}
\newcommand{\irchi}[2]{\raisebox{\depth}{$#1\chi$}}

\date{\empty}

\begin{document}

\title{The specification property for composition operators}
\author{Nilson C. Bernardes Jr., Emma D'Aniello, Martina Maiuriello}

\newcommand{\Addresses}{{
\bigskip
\footnotesize

\noindent
N.~C.~Bernardes~Jr.\\  
\textsc{Departamento de Matem\'atica Aplicada,\\
Universidade Federal do Rio de Janeiro,\\
Caixa Postal 68530, Rio de Janeiro, RJ, 21941-909, BRAZIL}\\
\textit{E-mail address: \em ncbernardesjr@gmail.com}

\medskip\noindent  
E.~D'Aniello,\\
\textsc{Dipartimento di Matematica e Fisica,\\ 
Universit\`a degli Studi della Campania ``Luigi Vanvitelli",\\
Viale Lincoln n. 5, 81100 Caserta, ITALIA}\\
\textit{E-mail address: \em emma.daniello@unicampania.it} 

\medskip\noindent
M.~Maiuriello,\\
\textsc{Dipartimento di Scienze Umane,\\ Universit\`a degli Studi ``Link Campus",\\
Via del Casale San Pio V, 44, 00165 Roma, ITALIA}\\
\textit{E-mail address: \em m.maiuriello@unilink.it} 

}}

\maketitle

\begin{abstract}
It is known that the operator specification property does not coincide, in general, with other notions of chaos in linear dynamics. 
In this paper, we show that, for a natural and widely studied class of linear operators, namely the dissipative composition operators of bounded distortion (including the class of weighted shifts), specification-type properties do not introduce new dynamical behaviors: 
they coincide with classical notions of chaos. 
We also extend these equivalences with conditions on periodic points and examine in detail the special case of the translation operator $T_{\alpha}$ acting on $L^p(\R,\B,\mu)$, where the measure $\mu$ is induced by a density function. 
Moreover, we show that, in general, there is no relationship between the operator specification property and the shadowing property in the linear context.
Finally, we analyse the important role played by the bounded distortion hypothesis.
\end{abstract}

\let\thefootnote\relax\footnote{2020 {\em Mathematics Subject Classification:} Primary: 47A16, 37B65; Secondary: 47B33, 47B37.\\
{\em Keywords:} Specification Property, Devaney Chaos, Frequent Hypercyclicity, Periodic Points, Composition Operators, Weighted Shifts.}

\begin{center}
\etocsetnexttocdepth{1} 
\etocsettocstyle{\small}{}
\tableofcontents
\end{center}

\section{Introduction}

Properties related to the concept of pseudotrajectory, such as the shadowing property and the specification property, 
play a fundamental role in the modern theory of dynamical systems and differential equations. 
A detailed study of the shadowing property for dissipative composition operators on $L^p$ spaces was conducted in the article 
\cite{DAnielloDarjiMaiuriello} by the second and third authors, in collaboration with Udayan Darji. 
Our main objective in this work is to carry out a detailed investigation of the specification property for this class of operators.

The Specification Property (SP) was originally introduced by Bowen \cite{Bowen1970,Bowen1971} 
in the context of Axiom~A diffeomorphisms.
It has been extensively investigated in the context of compact dynamical systems and has played a central role in topological dynamics, 
as it describes an intriguing mechanism for orbit approximation and is typically associated with chaotic behaviors: 
it implies topological mixing and positive topological entropy \cite{Bowen1970,Sigmund1974}, 
thus being a manifestation of chaotic behavior.

In the context of linear dynamics, where the space under consideration is typically non-compact and infinite-dimensional, 
the classical definition must be suitably adapted, leading to the notions of Operator Specification Property (OSP) and 
Operator Strong Specification Property (OSSP) for continuous linear operators acting on separable $F$-spaces,
introduced and investigated by Bartoll, Mart\'inez-Gim\'enez, Murillo-Arcila and Peris in the series of articles 
\cite{Bartoll2014,Bartoll2012,Bartoll2016}. 

In this article, we analyze the OSP and the OSSP for two fundamental classes of operators: 
bilateral weighted backward shifts on $c_0(\Z)$ and $\ell^p(\Z)$, 
and dissipative composition operators of bounded distortion on $L^p(X,\B,\mu)$ ($1 \leq p < \infty$).
The dynamics of composition operators and, in particular, weighted shifts, has been extensively investigated by many authors. 
We refer the reader to the recent articles \cite{AJK2025,BBP2026,GG2025}, where numerous references to previous works can be found.

Our first main result (Theorem~\ref{thmMAINBw}) establishes a complete chain of equivalences for bilateral weighted backward shifts
$B_w$ on sequence $F$-spaces, combining previously known results with new ones: 
it shows that a certain summability condition on the weight sequence is equivalent to Devaney chaos, the OSP, the OSSP, 
the existence of a nontrivial fixed point, and the existence of nontrivial periodic points of every period.
This general result is applied to weighted shifts on $\ell^p(\Z)$ (Corollary~\ref{CorBwlp}) and on $c_0(\Z)$ (Corollary~\ref{CorBwc0}). 
In the case of the spaces $\ell^p(\Z)$, we can add ``frequent hypercyclicity'' to the list of equivalences,
but this is not true in the case of the space $c_0(\Z)$ (Remark~\ref{Remarkc0}).
We also present an analogous result for unilateral weighted backward shifts on sequence $F$-spaces
(Theorem~\ref{thmMAINBwUnilateral}).

Next, we show that, in general, there is no relationship between the operator specification property and the shadowing property
in the setting of linear dynamics (Theorem~\ref{OSP-Shad}).

We then extend the previous analysis on weighted shifts to dissipative composition operators of bounded distortion on $L^p(X,\B,\mu)$, 
where $\mu$ is $\sigma$-finite and $1 \leq p < \infty$.
By exploiting their relation with associated bilateral weighted backward shifts via factor maps, 
we prove that the property of having the OSP or the OSSP is shared by the two operators, 
through the shift-like correspondence provided in Theorem \ref{thmshiftlike}. 
As a consequence, we obtain an analogous result for these composition operators, namely, Theorem \ref{thmMAINT_f}: 
Devaney chaos, frequent hypercyclicity, the OSP, the OSSP, the existence of a nontrivial fixed point, 
and the existence of nontrivial periodic points of every period are all equivalent to the measure $\mu$ being finite.
As a special case, under suitable assumptions, we provide a concrete characterization of specification-type behaviors 
for the translation operator $T_{\alpha}$ acting on $L^p(\R,\B,\mu)$, 
where the measure $\mu$ is induced by a density function (Corollary~\ref{thmMAINTa}). 

In \cite{Bartoll2016}, the authors show that there are operators defined on the separable Hilbert space which are 
topologically mixing, Devaney chaotic and frequently hypercyclic, but do not have the OSP. 
Our results, on the other hand, show that, for all the above mentioned classes of linear operators, 
specification-type properties do not generate new dynamical phenomena beyond those already captured 
by classical notions of linear chaos. 

Finally, we enrich the theory with examples and begin to investigate the importance of the bounded distortion hypothesis 
in the study of dynamical properties of dissipative composition operators. 

\section{Preliminaries}

Throughout the article, $\K$ denotes either the field $\R$ of real numbers or the field $\mathbb{C}$ of complex numbers,
$\Z$ denotes the ring of integers, $\N$ denotes the set of all positive integers, and $\N_0 = \N \cup \{0\}$.

All topological vector spaces (in particular, $F$-spaces, Fr\'echet spaces and Banach spaces) considered in this work 
are assumed to be over $\K$, unless otherwise specified.
By an {\em operator} on a topological vector space $X$, we mean a continuous linear map $T : X \to X$.
Such an operator $T$ is said to be {\em invertible} if it has a continuous inverse $T^{-1} : X \to X$.

Given a set $X$ and a map $f : X \to X$, recall that $x \in X$ is said to be a {\em periodic point} of $f$ if 
$f^n(x) = x$ for some $n \in \N$.
In this case, the smallest $q \in \N$ for which $f^q(x) = x$ is called the {\em period} of $x$.
We denote by $\PER_q(f)$ the set of all $x \in X$ such that $f^q(x) = x$,
and by $\Per_q(f)$ the set of all periodic points of $f$ with period $q$.
Periodic points with period $1$ are also called {\em fixed points}.

Let $T$ be an operator on a topological vector space $X$.
Recall that $T$ is said to be {\em topologically transitive} (resp.\ {\em topologically mixing}) if,
for any pair $U,V$ of nonempty open sets in $X$, there exists $n \in \N_0$ (resp.\ $n_0 \in \N_0$) such that 
$T^n(U) \cap V \neq \emptyset$ (resp.\ for all $n \geq n_0$).
Moreover, $T$ is said to be {\em Devaney chaotic} if it is topologically transitive and has a dense set of periodic points.
Finally, $T$ is said to be {\em frequently hypercyclic} if there exists $x \in X$ such that, for any nonempty open set $U$ in $X$,
\[
\ldens(\{n \in \N_0 : T^n(x) \in U\}) > 0,
\]
where $\ldens(A)$ denotes the {\em lower density} of the subset $A$ of $\N_0$, that is,
\[
\ldens(A) = \liminf_{n \to \infty} \frac{\card(A \cap [0,n])}{n+1}\cdot
\] 

Let us now recall the definition of the specification property in the metric space setting.

\begin{defn}[{\bf{Specification Property}}] \cite[Definition 1]{Bartoll2016} \label{defnSP}
A continuous map $f : X \to X$ on a compact metric space $(X,d)$ is said to have the \emph{specification property (SP)} if, 
for every $\eps > 0$, there exists a positive integer $N_{\eps}$ such that, 
for any integer $s \ge 2$, any set $\{x_1,\ldots,x_s\} \subset X$ and any integers
\[
0=  a_{1} \leq  b_1 < a_2 \le b_2 < \cdots < a_{s} \leq b_s  
\]
satisfying
\[ 
a_{r} - b_{r -1} \ge N_{\eps} \quad \text{for all } 2 \leq r \leq s, 
\]
there exists a point $x \in X$ such that, for each $1 \le r \le s$ and each $a_r \le i \le b_r$, the following condition holds:
\[ 
d\bigl(f^i(x), f^i(x_r)\bigr) < \eps. 
\]
If, moreover, 
\[ 
f^n(x) = x, \quad \text{where } n = N_{\eps}+ b_s,
\]
then $f$ is said to have the \emph{periodic specification property (PSP)}.  
\end{defn}

\begin{rmk} 
In \cite{Bartoll2016}, by SP the authors mean PSP. 
\end{rmk}

\begin{defn}[{\bf{$\epsilon$-traced specification}}] \cite[Definition 3.1]{Marcin2025} 
Let $X$ be a compact metric space with metric $d$.
Given a map $f : X \to X$, an interval $[a, b] \subset \N_0$ with $0 \leq a \leq b$, and a point $x \in X$, we write 
\[
f^{[a,b]}(x) 
\]
for the finite sequence $(f^{i}(x))_{a \leq i \leq b}$ and call it the \emph{orbit segment of $x$  over $[a, b]$}. 
Moreover, given a positive integer $N$, an \emph{$N$-spaced specification} is a sequence of $s \geq 2$ orbit segments 
${(f^{[a_{r}, b_{r}]}(x_{r}))}_{1 \leq r \leq s}$ such that $a_{r} - {b_{r-1}}  \geq  N$ for all $2 \leq r \leq s$.
Finally, given $\eps >0$, a specification ${(f^{[a_{r}, b_{r}]}(x_{r}))}_{1 \leq r \leq s}$ is said to be \emph{$\eps$-traced} if 
there exists $x \in X$ such that $d(f^{i}(x),f^{i}(x_{r})) < \eps$ for all $1 \leq r \leq s$ and $a_{r} \leq i \leq b_{r}$.  
\end{defn}

\begin{rmk} 
By using the above definition, one can rewrite Definition \ref{defnSP} in the following way, as in \cite[Definition 3.2]{Marcin2025}:
Let $X$ be a compact metric space with metric $d$ and $f : X \to X$ be a continuous map. 
Then $f$ has the {\em{specification property}} if, for every $\eps >0$, there exists $N_\eps \in \N$ such that 
every $N_{\eps}$-spaced specification ${(f^{[a_{r}, b_{r}]}(x_{r}))}_{1 \leq r \leq s}$ is $\eps$-traced by a point $x \in X$
(with $f^n(x) = x$, where $n = N_{\eps}+ b_s$, in the case of the {\em periodic specification property}). 
\end{rmk}

Let us recall that every system with the specification property is topologically mixing \cite{Sigmund1974} and, for many systems, 
such as systems on the interval $[0,1]$, the specification property is equivalent to topological mixing \cite{Buzzi, Ruette}.

Now, recall that an {\em $F$-space} is a complete metrizable topological vector space.
The class of $F$-spaces includes Fr\'echet spaces (complete metrizable locally convex spaces) and, in particular, Banach spaces. 
If $X$ is an $F$-space, then $X$ always admits a compatible complete translation-invariant metric $d$ such that
\[
\|x\| = d(x,0) \ \ \ \ \ (x \in X)
\]
defines an {\em $F$-norm} on $X$, which means that the following properties hold for each $x,y \in X$ and $\lambda \in \K$:
\begin{itemize}
\item $\|x + y\| \leq \|x\| + \|y\|$;
\item $\|\lambda x\| \leq \|x\|$ if $|\lambda| \leq 1$;
\item $\lim_{\lambda \to 0} \|\lambda x\| = 0$;
\item $\|x\| = 0$ implies $x = 0$.
\end{itemize}

In the next two definitions, $X$ denotes a separable $F$-space endowed with an $F$-norm $\|\cdot\|$
and $T$ denotes an operator on $X$.

The following definition is a natural translation of Definition~\ref{defnSP} to the context of linear dynamics. 

\begin{defn}[{\bf{Operator Strong Specification Property}}] \label{defnSSP} \cite[Definition~1]{Bartoll2014} 
We say that $T$ has the \emph{operator strong specification property (OSSP)} if there exists an increasing sequence $(K_m)_{m \in \N}$ 
of $T$-invariant compact sets with $0 \in K_1$ and $\overline{\bigcup_{m \in \N} K_m} = X$ such that, 
for each $m \in \N$, the map $T|_{K_m}$ has the PSP, that is, for every $\eps > 0$, there exists a 
positive integer $N_{m,\eps}$ such that, for any integer $s \ge 2$, any set $\{x_1, \ldots, x_s\} \subset K_m$ and any integers
\[ 
0 = a_1 \le b_1 < a_2 \le b_2 < \cdots < a_s \le b_s 
\]
with $a_{r} - b_{r-1} \ge N_{m,\eps}$ for all $2 \le r \le s$, there exists a point $x \in K_m$ such that, for each $1 \le r \le s$ and
each $a_r \le i \le b_r$, the following conditions hold:
\[ 
\|T^i(x) - T^i(x_r)\| < \eps, 
\]
\[ 
T^n(x) = x, \quad \text{where } n = N_{m,\eps} + b_s. 
\] 
\end{defn}

\begin{defn}[{\bf{Operator Specification Property}}] \cite[Definition 3]{Bartoll2016}  
By removing the compactness assumption of each $K_m$ in Definition~\ref{defnSSP}, 
the operator $T$ is said to have the \emph{operator specification property (OSP)}.
\end{defn}

\begin{rmk} 
As observed in \cite[Observation 4]{Bartoll2016}, 
``it is hard to think of a map having the specification property outside of the compact setting; 
in other words, for all cases we know of operators having the OSP, the required sets  $K_m$ are always compact".  
\end{rmk}

What is the idea behind the above definitions? 
Roughly speaking, the OSP expresses a strong approximation property for the dynamics of $T$. 
It means that the space $X$ can be exhausted by an increasing sequence of $T$-invariant sets $(K_m)_{m \in \N}$ 
on which $T$ satisfies the PSP, meaning that, for each $m \in \N$ and each $\eps >0$, there exists $N_{m,\eps} \in \N$ such that 
\begin{enumerate}
\item[{}]- given finitely many points $x_1,\dots,x_s \in K_m$, $s \geq 2$,
\item [{}] - given the ``time intervals'' $[a_r,b_r]$, with large enough gaps between them, 
  that is, $a_{r} - b_{r-1} \ge N_{m,\eps}$ for all $2 \le r \le s$,
\end{enumerate}
then, one can find a single point $x \in K_m$ whose orbit $\eps$-shadows the orbit of each point $x_r$ 
on the corresponding time interval $[a_r, b_r]$, for $1\leq r \leq s$, that is, $\|T^i(x) - T^i(x_r)\| < \eps$,  for each $a_r \le i \le b_r$,
and the point $x$ is periodic, with $T^{N_{m,\eps} + b_s}(x) = x$.

This last point, of $x$ being periodic, is an important feature of all the above definitions. 
Not only one can approximate parts of the orbits of arbitrary finite collections of points with large enough gaps, 
but one can do it by using periodic orbits. 
This shows recurrence and complexity in the dynamics of the operator $T$, typically associated with strong chaotic behaviors. 
Indeed, operators with the OSP have positive topological entropy (see \cite[Page~919]{Brian2017} and \cite[Theorem~B]{LSV2026}) 
and are topologically mixing, frequently hypercyclic and Devaney chaotic (see \cite[Section 3]{Bartoll2016}).

We recall that a system $(Y, S)$ is a {\textit{factor}} of a system $(X,T)$ if there exists a continuous surjection $\Pi: X \rightarrow Y$ such that $\Pi \circ T = S \circ \Pi$. Such a map $\Pi$ is called a {\em{quasi-conjugation}}.

Some basic properties of the OSP and the OSSP are collected in \cite[Section 2]{Bartoll2016}, like, for instance, 
the facts that if an operator $T$ has the OSP (resp.\ the OSSP), then so does $T^k$ for each $k \in \N$ \cite[Proposition 8]{Bartoll2016}, 
and that the OSP (resp.\ the OSSP) is preserved by uniformly continuous quasi-conjugations \cite[Proposition 6]{Bartoll2016}.

\section{The OSP and the OSSP for weighted shifts}

The study of the specification property in linear dynamics was initiated in \cite{Bartoll2012},
where the concept of OSSP for continuous linear operators on separable Banach spaces was introduced, and 
characterizations of this concept were obtained for weighted backward shifts 
on the classical sequence spaces $c_0$ and $\ell^p$ ($1 \leq p < \infty$).
This study was extended to separable $F$-spaces and sequence $F$-spaces in \cite{Bartoll2014}.
A more extensive study of the OSP and the OSSP for operators on separable $F$-spaces was developed in \cite{Bartoll2016}, 
where several interesting properties were obtained, 
including the fact that the OSP implies topological mixing, frequent hypercyclicity and Devaney chaos.

In this section, we complement the characterizations of the OSSP for weighted backward shifts on sequence $F$-spaces 
presented in \cite{Bartoll2014}, by performing a careful analysis of the periodic points of these weighted shifts. 
In particular, we show that the OSSP is equivalent to the existence of a nontrivial fixed point, which, in turn, 
implies the existence of periodic points of every period.

Recall that a {\em sequence $F$-space} ({\em over $\Z$}) is an $F$-space $X$ which is a vector subspace of the product space $\K^\Z$
such that the inclusion map $X \to \K^\Z$ is continuous, that is, convergence in $X$ implies coordinatewise convergence.
Given a sequence $w = (w_n)_{n \in \Z}$ of nonzero scalars, it follows from the closed graph theorem that
the {\em bilateral weighted backward shift}
\[
B_w((x_n)_{n \in \Z}) = (w_{n+1}x_{n+1})_{n \in \Z}
\]
is a continuous linear operator on $X$ whenever it maps $X$ into itself.
The {\em canonical vectors} of $\K^\Z$ are the vectors $e_n$, $n \in \Z$, where the $n^\text{th}$ coordinate of $e_n$ is $1$
and the other coordinates of $e_n$ are $0$. 
The sequence $(e_n)_{n \in \Z}$ is a {\em basis} of $X$ if each $e_n$ belongs to $X$ and
\[
x = \sum_{n=-\infty}^\infty x_ne_n \ \ \text{ for all } x = (x_n)_{n \in \Z} \in X.
\]
If, in addition, for each $x = (x_n)_{n \in \Z} \in X$, the series 
\[
\sum_{n=1}^\infty x_{-n}e_{-n} \ \ \ \text{ and } \ \ \ \sum_{n=1}^\infty x_n e_n
\]
converge unconditionally, then we say that $(e_n)_{n \in \Z}$ is an {\em unconditional basis} of $X$.

\begin{thm}[Equivalences for $B_w$]\label{thmMAINBw}
For any bilateral weighted backward shift $B_w : X \to X$ on a sequence $F$-space $X$ in which the sequence
$(e_n)_{n \in \Z}$ of canonical vectors is an unconditional basis, the following conditions are equivalent: 
\begin{itemize}
\item [\rm (i)] The series
  \[
  \sum_{n = -\infty}^0 \Big(\prod_{i=n+1}^{0} w_i\Big) e_n + \sum_{n=1}^\infty \Big(\prod_{i=1}^{n} w_i\Big)^{-1} e_n
  \]
  converges in $X$;
\item [\rm (ii)] $B_w$ is Devaney chaotic;
\item [\rm (iii)] $B_w$ has the OSP;
\item [\rm (iv)] $B_w$ has the OSSP;
\item [\rm (v)] $B_w$ has a nontrivial periodic point; 
\item [\rm (vi)] $B_w$ has a nontrivial fixed point;
\item [\rm (vii)] for each $q \in \N$, $B_w$ has a nontrivial periodic point with period $q$.
\end{itemize}
Moreover, in this case, $\Per_q(B_w)$ is a dense open set in $\PER_q(B_w)$ for every $q \in \N$.
In particular, $\Per_q(B_w)$ is dense in $\Per_t(B_w)$ whenever $q,t \in \N$ and $t$ divides $q$.
\end{thm}

\begin{proof}
The equivalences (i) $\Leftrightarrow$ (ii) $\Leftrightarrow$ (v) are contained in \cite[Theorem~9]{Grosse2000}
(see also \cite[Section~4.1]{GrosseErdmannPeris2011}).
The inclusion of (iv) among these equivalences comes from \cite[Theorem~3]{Bartoll2014}.
On the other hand, (iii) $\Rightarrow$ (ii) follows from \cite[Proposition~12]{Bartoll2016} and it is trivial that (iv) $\Rightarrow$ (iii).
It is also trivial that (vii) $\Rightarrow$ (vi) $\Rightarrow$ (v).
Therefore, to obtain all equivalences, it remains to prove that (i) $\Rightarrow$ (vii).
For this purpose, fix $q \in \N$.
Choose any nonzero scalar $x_0 \in \K$.
If $q \geq 2$, choose also scalars $x_1,\ldots,x_{q-1} \in \K$ so that 
\begin{equation}\label{FS-Eq1}
x_1 \neq \frac{x_0}{w_1}\,, \ \ x_2 \neq \frac{x_0}{w_1 w_2}\,, \ \ldots \ , x_{q-1} \neq \frac{x_0}{w_1 \cdots w_{q-1}}\cdot
\end{equation}
For each $s \in \{0,\ldots,q-1\}$, the unconditional convergence of the series in (i) implies that the subseries
\[
\sum_{n = -\infty}^{-1} \Big(\prod_{i=nq+s+1}^{0} w_i\Big) e_{nq+s} \ \ \text{ and } \ \ 
\sum_{n=1}^\infty \Big(\prod_{i=1}^{nq+s} w_i\Big)^{-1} e_{nq+s}
\]
of the first and the second series in (i), respectively, converge in $X$.
Therefore,
\begin{equation}\label{FS-Eq2}
x = \sum_{s=0}^{q-1} \Bigg[ \sum_{n = -\infty}^{-1} \Big(\prod_{i=nq+s+1}^{s} w_i\Big) x_s e_{nq+s} + x_s e_s
+ \sum_{n=1}^\infty \Big(\prod_{i=s+1}^{nq+s} w_i\Big)^{-1} x_s e_{nq+s}\Bigg]
\end{equation}
defines an element of $X$.
We have that
\begin{align*}
(B_w)^q(x) 
&= \sum_{s=0}^{q-1} \Bigg[ \sum_{n = -\infty}^{-1} \Big(\prod_{i=(n-1)q+s+1}^{s} w_i\Big) x_s e_{(n-1)q+s}
     + \Big(\prod_{i=-q+s+1}^{s} w_i\Big) x_s e_{-q+s}\\
& \ \ \ \ \ \ \ \ \ \ \ \ + x_s e_s + \sum_{n=2}^\infty \Big(\prod_{i=s+1}^{(n-1)q+s} w_i\Big)^{-1} x_s e_{(n-1)q+s}\Bigg]\\
&= \sum_{s=0}^{q-1} \Bigg[ \sum_{n = -\infty}^{-2} \Big(\prod_{i=nq+s+1}^{s} w_i\Big) x_s e_{nq+s}
     + \Big(\prod_{i=-q+s+1}^{s} w_i\Big) x_s e_{-q+s}\\
& \ \ \ \ \ \ \ \ \ \ \ \ + x_s e_s + \sum_{n=1}^\infty \Big(\prod_{i=s+1}^{nq+s} w_i\Big)^{-1} x_s e_{nq+s}\Bigg]\\
&= x.
\end{align*}
Moreover, if $q \geq 2$ then (\ref{FS-Eq1}) implies that $(B_w)^j(x) \neq x$ for each $j \in \{1,\ldots,q-1\}$.
Thus, $x$ is a nontrivial periodic point of $B_w$ with period $q$. 

To prove the density of $\Per_q(B_w)$ in $\PER_q(B_w)$, 
we assume $q \geq 2$, take $y = (y_n)_{n \in \Z} \in \PER_q(B_w)$ and fix $\eps > 0$. 
We have to find $x \in \Per_q(B_w)$ such that $\|x - y\| < \eps$.
By the unconditional convergence of the series in (i), there exists $m \geq 2$ such that
\begin{equation}\label{FS-Eq3}
\Big\|\sum_{n \in J} \Big(\prod_{i=n+1}^{0} w_i\Big) e_n\Big\| < \frac{\eps}{3q} \ \ \text{ and } \ \
\Big\|\sum_{n \in H} \Big(\prod_{i=1}^{n} w_i\Big)^{-1} e_n\Big\| < \frac{\eps}{3q}\,,
\end{equation}
whenever $J \subset \{n \in -\N : n \leq -(m-1)q\}$ and $H \subset \{n \in \N : n \geq mq\}$ are finite sets.
On the other hand, since $(B_w)^q(y) = y$, we have that
\[
y_{k+q} = \dfrac{y_k}{w_{k+1}\cdots w_{k+q}} \ \ \text{ for all } k \in \Z.
\]
In particular, for $k = s$, with $s \in \{0,\ldots,q-1\}$,
\[
y_{q+s} = \frac{y_s}{\Pi_{i=s+1}^{q+s} w_i}\,, 
\]
for $k = q + s$, with $s \in \{0,\ldots,q-1\}$,
\[
y_{2q+s} = \frac{y_{q+s}}{\Pi_{i=q+s+1}^{2q+s} w_i} =   \frac{y_s}{\Pi_{i=s+1}^{2q+s} w_i }\,, 
\]
and so on. Thus, for any $n \in \N$ and $s \in \{0,\ldots,q-1\}$,
\begin{equation}\label{FS-Eq4}
y_{nq+s} = \frac{y_s}{\Pi_{i=s+1}^{nq+s} w_i}\cdot
\end{equation}
Now, for $k = -q+s$, with  $s \in \{0,\ldots,q-1\}$, 
\[
y_s = \frac{y_{-q+s}}{\Pi_{i = -q+s+1}^s w_i},\ \ \text{that is,} \ \ y_{-q+s} = y_s \Pi_{i = -q+s+1}^{s} w_i,
\]
for $k = -2q+s$, with $s \in \{0,\ldots,q-1\}$,
\[
y_{-q+s} = \frac{y_{-2q+s}}{\Pi_{i=-2q+s+1}^{-q+s} w_i}, \ \ \text{that is,} \ \ 
y_{-2q+s} = y_{-q+s} \Pi_{i=-2q+s+1}^{-q+s} w_i = y_s \Pi_{i=-2q+s+1}^{s} w_i, 
\]
and so on. Thus, for any $n \in \N$ and $s \in \{0,\ldots,q-1\}$,
\begin{equation}\label{FS-Eq5}
y_{-nq+s} = y_s \Pi_{i=-nq+s+1}^{s} w_i. 
\end{equation}
Given any $\delta > 0$, we can choose $x_0,\ldots,x_{q-1} \in \K$ such that $x_0 \neq 0$, (\ref{FS-Eq1}) holds and
\[
|x_s - y_s| < \delta \ \ \text{ for all } s \in \{0,\ldots,q-1\}.
\]
Define $x$ by the formula given in (\ref{FS-Eq2}). As we saw in the first paragraph, $x \in \Per_q(B_w)$.
Moreover, by (\ref{FS-Eq4}) and (\ref{FS-Eq5}),
\begin{align}
x - y = \sum_{s=0}^{q-1} \Bigg[ &\sum_{n = -\infty}^{-1} \Big(\prod_{i=nq+s+1}^{s} w_i\Big) (x_s - y_s) e_{nq+s} \notag\\
  & \ + (x_s - y_s) e_s + \sum_{n=1}^\infty \Big(\prod_{i=s+1}^{nq+s} w_i\Big)^{-1} (x_s - y_s) e_{nq+s}\Bigg].\label{FS-Eq6}
\end{align}
By choosing $\delta > 0$ small enough, we have that
\[
\Big|(x_s - y_s) \Big(\prod_{i=1}^s w_i\Big)\Big| < 1 \ \ \text{ for all } s \in \{0,\ldots,q-1\}.
\]
Hence, it follows from the first inequality in (\ref{FS-Eq3}) that
\begin{equation}\label{FS-Eq7}
\Big\|\sum_{n = -\infty}^{-m} \Big(\prod_{i=nq+s+1}^{s} w_i\Big) (x_s - y_s) e_{nq+s}\Big\| 
  \leq \Big\|\sum_{n = -\infty}^{-m} \Big(\prod_{i=nq+s+1}^{0} w_i\Big) e_{nq+s}\Big\| \leq \frac{\eps}{3q}\cdot
\end{equation}
Similarly, it follows from the second inequality in (\ref{FS-Eq3}) that
\begin{equation}\label{FS-Eq8}
\Big\|\sum_{n=m}^\infty \Big(\prod_{i=s+1}^{nq+s} w_i\Big)^{-1} (x_s - y_s) e_{nq+s}\Big\|
    \leq \Big\|\sum_{n=m}^\infty \Big(\prod_{i=1}^{nq+s} w_i\Big)^{-1} e_{nq+s}\Big\| \leq \frac{\eps}{3q}\cdot
\end{equation}
Finally, by choosing $\delta$ small enough, we can also guarantee that
\begin{align}
\Big\| &\sum_{n = -m+1}^{-1} \Big(\prod_{i=nq+s+1}^{s} w_i\Big) (x_s - y_s) e_{nq+s} \notag\\
  & \ + (x_s - y_s) e_s + \sum_{n=1}^{m-1} \Big(\prod_{i=s+1}^{nq+s} w_i\Big)^{-1} (x_s - y_s) e_{nq+s}\Big\| < \frac{\eps}{3q}\,,\label{FS-Eq9}
\end{align}
for each $s \in \{0,\ldots,q-1\}$.
In view of (\ref{FS-Eq6})--(\ref{FS-Eq9}), we conclude that $\|x - y\| < \eps$.

Finally, it is clear that $\Per_q(B_w)$ is open in $\PER_q(B_w)$ and that the last assertion of the theorem holds.
\end{proof}

As an application, let us consider the case of the classical Banach sequence spaces $\ell^p(\Z)$ for $1 \leq p < \infty$.
In this case, the concept of frequent hypercyclicity can be included among the equivalences, as already observed in \cite{Bartoll2014}.
Indeed, this follows from the fact that the concepts of frequent hypercyclicity and Devaney chaos coincide 
for weighted shifts on $\ell^p(\Z)$ (see \cite{BayartRuzsa}).

\begin{cor}[Equivalences for $B_w$ on $\ell^p(\Z)$]\label{CorBwlp}
For any bilateral weighted backward shift $B_w$ on $\ell^p(\Z)$ ($1 \leq p < \infty$), the following conditions are equivalent: 
\begin{itemize}
\item [\rm (i)] $\displaystyle \sum_{n=1}^{\infty} \frac{1}{\prod_{i=1}^{n} |w_i|^p} < \infty
  \ \text{ and } \ \sum_{n=1}^{\infty} \prod_{i=-n+1}^{0} |w_i|^p < \infty$;
\item [\rm (ii)] $B_w$ is Devaney chaotic;
\item [\rm (iii)] $B_w$ is frequently hypercyclic;
\item [\rm (iv)] $B_w$ has the OSP;
\item [\rm (v)] $B_w$ has the OSSP;
\item [\rm (vi)] $B_w$ has a nontrivial periodic point; 
\item [\rm (vii)] $B_w$ has a nontrivial fixed point;
\item [\rm (viii)] for each $q \in \N$, $B_w$ has a nontrivial periodic point with period $q$.
\end{itemize}
\end{cor}

\smallskip
In the case of the space $c_0(\Z)$, we have the following result.

\begin{cor}[Equivalences for $B_w$ on $c_0(\Z)$]\label{CorBwc0}
For any bilateral weighted backward shift $B_w$ on $c_0(\Z)$, the following conditions are equivalent: 
\begin{itemize}
\item [\rm (i)] $\displaystyle \lim_{n \to \infty} \frac{1}{\prod_{i=1}^{n} w_i} = \lim_{n \to \infty} \prod_{i=-n+1}^{0} w_i = 0$;
\item [\rm (ii)] $B_w$ is Devaney chaotic;
\item [\rm (iii)] $B_w$ has the OSP;
\item [\rm (iv)] $B_w$ has the OSSP;
\item [\rm (v)] $B_w$ has a nontrivial periodic point; 
\item [\rm (vi)] $B_w$ has a nontrivial fixed point;
\item [\rm (vii)] for each $q \in \N$, $B_w$ has a nontrivial periodic point with period $q$.
\end{itemize}
\end{cor}

\begin{rmk}\label{Remarkc0}
Unlike the case of weighted shifts on $\ell^p$ spaces, 
we cannot include the item ``$B_w$ is frequently hypercyclic'' in the previous corollary.
In fact, Bayart and Ruzsa constructed in \cite[Section~6]{BayartRuzsa} 
a frequent hypercyclic invertible weighted shift $B_w$ on $c_0(\Z)$ such that
\[
\prod_{i=-n+1}^{0} w_i = 1 \ \ \text{ for all } n \in A,
\]
where $A$ is a certain subset of $\N$ with positive lower density.
In particular, condition (i) fails, although $B_w$ is frequent hypercyclic.
However, the equivalent conditions in Corollary~\ref{CorBwc0} imply that $B_w$ is frequent hypercyclic.
This follows from the fact that the OSP implies frequent hypercyclicity \cite[Theorem~13]{Bartoll2016}.
\end{rmk}

\begin{rmk}
Of course, in Corollaries \ref{CorBwlp} and \ref{CorBwc0}, it is also true that 
$\Per_q(B_w)$ is a dense open set in $\PER_q(B_w)$ and that $\Per_q(B_w)$ is dense in $\Per_t(B_w)$ whenever $t$ divides $q$.
\end{rmk}

Recall that a {\em sequence $F$-space} ({\em over $\N$}) is an $F$-space $X$ which is a vector subspace of the product space $\K^\N$
such that convergence in $X$ implies coordinatewise convergence.
As before, if $w = (w_n)_{n \in \N}$ is a sequence of nonzero scalars, 
it follows from the closed graph theorem that the {\em unilateral weighted backward shift}
\[
B_w(x_1,x_2,x_3,\ldots) = (w_2 x_2,w_3 x_3,\ldots)
\]
is a continuous linear operator on $X$ whenever it maps $X$ into itself.
By abuse of language, we also denote the {\em canonical vectors} of $\K^\N$ by $e_n = (\delta_{n,k})_{k \in \N}$, $n \in \N$.
The sequence $(e_n)_{n \in \N}$ is a {\em basis} of $X$ if each $e_n$ belongs to $X$ and
\[
x = \sum_{n=1}^\infty x_n e_n \ \ \text{ for all } x = (x_n)_{n \in \N} \in X.
\]
If, in addition, the above series always converges unconditionally, then we say that $(e_n)_{n \in \N}$ is an {\em unconditional basis} of $X$.

We have the following version of Theorem~\ref{thmMAINBw} for unilateral weighted backward shifts on sequence $F$-spaces,
which complements the equivalences presented in \cite[Theorem~2]{Bartoll2014}.

\begin{thm}\label{thmMAINBwUnilateral}
For any unilateral weighted backward shift $B_w : X \to X$ on a sequence $F$-space $X$ in which the sequence
$(e_n)_{n \in \N}$ of canonical vectors is an unconditional basis, the following conditions are equivalent: 
\begin{itemize}
\item [\rm (i)] The series
  \[
  \sum_{n=1}^\infty \Big(\prod_{i=1}^{n} w_i\Big)^{-1} e_n
  \]
  converges in $X$;
\item [\rm (ii)] $B_w$ is Devaney chaotic;
\item [\rm (iii)] $B_w$ has the OSP;
\item [\rm (iv)] $B_w$ has the OSSP;
\item [\rm (v)] $B_w$ has a nontrivial periodic point; 
\item [\rm (vi)] $B_w$ has a nontrivial fixed point;
\item [\rm (vii)] for each $q \in \N$, $B_w$ has a nontrivial periodic point with period $q$.
\end{itemize}
Moreover, in this case, $\Per_q(B_w)$ is a dense open set in $\PER_q(B_w)$ for every $q \in \N$.
In particular, $\Per_q(B_w)$ is dense in $\Per_t(B_w)$ whenever $q,t \in \N$ and $t$ divides $q$.
\end{thm}

The proof of the above theorem is similar to, but slightly simpler than, that of Theorem~\ref{thmMAINBw}.
Results analogous to Corollaries~\ref{CorBwlp} and~\ref{CorBwc0} are valid for unilateral weighted backward shifts
on $\ell^p(\N)$ ($1 \leq p < \infty$) and $c_0(\N)$, respectively.
However, we leave all the details up to the reader.

\section{The OSP and the OSSP versus the shadowing property}

Given a metric space $X$ with metric $d$, recall that a homeomorphism $f : X \to X$ is said to have the {\em shadowing property} if,
for every $\eps > 0$, there exists $\delta > 0$ such that for any bilateral sequence $(x_j)_{j \in \Z}$ in $X$ satisfying
\[
d(f(x_j),x_{j+1}) \leq \delta \ \text{ for all } j \in \Z
\]
(called a {\em $\delta$-pseudotrajectory} of $f$), there exists a point $x \in X$ whose true trajectory {\em $\eps$-shadows} 
$(x_j)_{j \in \Z}$, in the sense that
\[
d(x_j,f^j(x)) < \eps \ \text{ for all } j \in \Z.
\]
This is one of the key concepts in the modern theory of dynamical systems 
and has been extensively investigated by various authors in different contexts for over five decades.
Recent articles on the shadowing property in the setting of linear dynamics include \cite{ABM21,BCDFP25,BCDMP18,BeMe21,BePe24,Pit}.

It is natural to ask if there is any relationship between the specification property and the shadowing property in the context of linear dynamics.
The goal of this section is to answer this question through the following result.

\begin{thm}\label{OSP-Shad}
Let $X = c_0(\Z)$ or $X = \ell^p(\Z)$ for some $1 \leq p < \infty$.
\begin{itemize}
\item [\rm (a)] There exist invertible weighted shifts $B_w$ on $X$ that have the operator strong specification property, 
  but do not have the shadowing property. 
\item [\rm (b)] There exist invertible weighted shifts $B_w$ on $X$ that have the shadowing property, 
  but do not have the operator specification property. 
\end{itemize}
\end{thm}

\begin{proof}
(a): Fix a real number $0 < a < 1$. The weight sequence $w = (w_n)_{n \in \Z}$ will be defined 
by concatenating certain blocks $B_n$ with the numbers $a$ and $a^{-1}$ in a suitable manner. 
More specifically, for each $n \in \N$, let $B_n$ denote the block $1,1,\ldots,1$ consisting of $n$ numbers $1$, and define
\[
w = (\ldots,B_4,a,B_3,a,B_2,a,B_1,a,a^{-1},B_1,a^{-1},B_2,a^{-1},B_3,a^{-1},B_4,\ldots),
\]
where the first $a^{-1}$ appears at position $1$. Since
\[
\sum_{n=1}^{\infty} \frac{1}{\prod_{i=1}^{n} |w_i|} = \sum_{n=1}^{\infty} \prod_{i=-n+1}^{0} |w_i| = \sum_{n=1}^\infty (n+1) a^{n},
\]
which is a convergent series (by the root test), it follows from Corollaries~\ref{CorBwlp} and~\ref{CorBwc0} that $B_w$ has the OSSP.
In order to prove that $B_w$ does not have the shadowing property, 
we shall use the characterization of the shadowing property for weighted shifts obtained in \cite[Theorem~18]{BeMe21}, namely:
$B_w$ has the shadowing property if and only if one of the following conditions holds:
\begin{itemize}
\item $\displaystyle \lim_{n \to \infty} \sup_{k \in \Z} |w_k \cdots w_{k+n}|^\frac{1}{n} < 1$;
\item $\displaystyle \lim_{n \to \infty} \inf_{k \in \Z} |w_k \cdots w_{k+n}|^\frac{1}{n} > 1$;
\item $\displaystyle \lim_{n \to \infty} \sup_{k \in \N} |w_{-k-n} \cdots w_{-k}|^\frac{1}{n} < 1$
          and
          $\displaystyle \lim_{n \to \infty} \inf_{k \in \N} |w_k \cdots w_{k+n}|^\frac{1}{n} > 1$.
\end{itemize}
Since the weight sequence $w$ contains arbitrarily large blocks of 1's on both sides
(i.e., with negative indices and with positive indices), we see that none of the above conditions can be satisfied, 
proving that $B_w$ does not have the shadowing property.

\smallskip\noindent
(b): It is known that if an invertible operator $T$ on a Banach space $Y$ is hyperbolic (that is, its spectrum does not intersect the unit circle),
then $T$ has the shadowing property \cite[Theorem~1]{Omb94}, but it is not even Li-Yorke chaotic \cite[Theorem~C]{BCDMP18}, 
which is the weakest of the usual notions of chaos in linear dynamics. 
Since operators with the OSP are Devaney chaotic \cite[Proposition~12]{Bartoll2016}, we conclude that $T$ cannot have the OSP.
Thus, any invertible hyperbolic weighted shift $B_w$ on $X$ does the job
(a characterization of the hyperbolic $B_w$'s can be found in \cite[Remark~35(b)]{BCDMP18}).
\end{proof}

\section{The OSP and the OSSP for composition operators}\label{s-OSP,Tf}

Our goal in this section is to investigate the two properties, OSP and OSSP, 
in the context of composition operators defined on $L^p$ spaces.

Given a measure space $(X,\B,\mu)$, recall that a transformation $f : X \to X$ is said to be
\begin{itemize}
\item {\em bimeasurable} if $f(B) \in \B$ and $f^{-1}(B) \in \B$ for all $B \in \B$;
\item {\em nonsingular} if $\mu(f^{-1}(B)) = 0$ if and only if $\mu(B) = 0$.
\end{itemize}
Recall also that a {\em measurable system} is a quadruple $(X,\B,\mu,f)$, where
\begin{itemize}
\item $(X,\B,\mu)$ is a $\sigma$-finite measure space with $\mu(X) > 0$;
\item $f : X \to X$ is an injective bimeasurable nonsingular transformation;
\item there is a constant $c > 0$ such that
\begin{equation}\label{condition}
\mu(f^{-1}(B)) \leq c\, \mu(B) \ \textrm{ for every } B \in \B. \tag{$\star$}
\end{equation}
\end{itemize}
If $f$ is bijective and $f^{-1}$ also satisfies ($\star$), we say that the measurable system is {\em invertible}.

If $(X,\B,\mu,f)$ is a measurable system and $p \in [1,\infty)$, 
it is well known that ($\star$) guarantees that the {\em composition operator} 
\[
T_f : \varphi \in L^p(X,\B,\mu) \mapsto \varphi \circ f \in L^p(X,\B,\mu)
\]
is a well-defined bounded linear operator.
Moreover, if the measurable system is invertible, then the composition operator $T_f$ is invertible and $T_f^{-1} = T_{f^{-1}}$. 

Let $(X,\B,\mu,f)$ be a measurable system.
If there exists $W \in \B$, with $0 < \mu(W) < \infty$, such that 
\[
X =\dot\bigcup_{k \in \Z}\, f^k (W) \ \ \ \ \ (\text{disjoint union}),
\]
then we say that $(X,\B,\mu,f)$ is a {\em dissipative measurable system} (and $T_f$ is a {\em dissipative composition operator})
{\em generated by $W$}, and $W$ is called a {\em wandering set}.
If, in addition, there exists $K > 0$ such that
\[
 \dfrac{1}{K}\, \mu(f^k(W))\mu(B) \leq \mu(f^k (B))\mu (W) \leq K \mu(f^k(W))\mu(B)
\]
for all $k \in \Z$ and $B \in \B(W)$, where $\B(W) = \{B \cap W : B \in \B\}$, 
then the dissipative measurable system (and the dissipative composition operator) is said to have {\em bounded distortion}.
It is known that for such a composition operator $T_f$, the concepts of Devaney chaos and frequent hypercyclicity coincide \cite[Theorem~3.7]{Darjipires}. 
Therefore, it is natural to ask whether a more general version of Corollary~\ref{CorBwlp} holds 
for dissipative composition operators $T_f$ of bounded distortion. 

In the sequel, we fix a real number $p \in [1,\infty)$ and an invertible dissipative measurable system $(X,\B,\mu,f)$, 
unless otherwise specified.
Moreover, we will usually denote $L^p(X,\B,\mu)$ simply by $L^p(X)$.

Each weighted backward shift $B_w$ on $\ell^p(\Z)$, with weights in $(0,\infty)$ (we can always reduce to positive weights, 
see, for instance, \cite[Proposition 6.2]{Conway} and \cite[Proposition 2.21]{DAnielloMaiuriello2023}), 
can be seen as a specific composition operator $T_g$ on $L^p(\Z,{\mathcal P}(\Z),\nu)$, 
where ${\mathcal P}(\Z)$ is the power set of $\Z$, $g: \Z \to \Z$ is the $+1$-map, and  
\begin{align*}
\nu(\{k\})& = \left\{
\begin{array}{cl}
 \dfrac{1}{(w_1\cdots w_k)^{p}} & \text{ if } k > 0 \\
 1 &  \text{ if } k=0 \\
 \left( w_{k+1}\cdots w_0\right)^{p} &  \text{ if } k < 0
 \end{array}\right.
 \end{align*}
(see \cite[Proposition~2.1]{DAnielloDarjiMaiuriello2}).

It is known from \cite[Lemma~4.2.3]{DAnielloDarjiMaiuriello} that, given a dissipative composition operator $T_f$ 
of bounded distortion on $L^p(X)$, the bilateral weighted backward shift $B_w$ on $\ell^p(\Z)$ with weights 
\[
w_{k} = \left(\frac{\mu(f^{k-1}(W))}{\mu(f^{k}(W))}\right)^{\frac{1}{p}} \ \ \ \ \ (k \in \Z)
\]
is a factor of $T_f$ by means of a {\em linear} quasi-conjugation $\Gamma : L^p(X) \to \ell^p(\Z)$. 
Clearly, the weighted shift $B_w$ associated with the composition operator $T_f$ is invertible. 

The first approach to the study of the OSP and the OSSP for $T_f$ is the analysis of its shift-like behavior 
(see \cite{DAnielloDarjiMaiuriello2} for the background), as the following result shows.

\begin{thm}[Shift-like behavior] \label{thmshiftlike}
Let $T_f$ be a dissipative composition operator of bounded distortion on $L^p(X)$ 
and $B_w$ the associated bilateral weighted backward shift on $\ell^p(\Z)$. Then, 
 \begin{itemize}
     \item[\rm (a)] $T_f$ has the OSP if and only if $B_w$ has the OSP;
     \item[\rm (b)] $T_f$ has the OSSP if and only if $B_w$ has the OSSP.
 \end{itemize}
\end{thm}
\begin{proof}
$((a), (b) \Rightarrow)$. It follows from the facts that $B_w$ is a factor of $T_f$ by means of a linear quasi-conjugation
and that the OSP and the OSSP are preserved by uniformly continuous quasi-conjugations.

\noindent
$((a), (b) \Leftarrow)$. To show this implication, assume that $B_w$ has the OSP (resp.\ the OSSP). 
By Corollary~\ref{CorBwlp}, this is equivalent to $B_w$ being frequently hypercyclic, meaning, 
by \cite[Theorem M]{DAnielloDarjiMaiuriello2}, that $T_f$ is frequently hypercyclic too. 
Hence, by \cite[Theorem 4.2]{Maiuriello2023}, $T_f$ satisfies the Frequent Hypercyclicity Criterion,
which implies that $T_f$ has the OSSP (hence, the OSP) by \cite[Theorem~4]{Bartoll2014}.
\end{proof}

We can now obtain a result analogous to Corollary~\ref{CorBwlp} for $T_f$.

\begin{thm}[Equivalences for $T_f$] \label{thmMAINT_f}
Let $T_f$ be a dissipative composition operator of bounded distortion on $L^p(X)$. Then, the following conditions are equivalent:
\begin{itemize}
\item [\rm (i)] $\mu(X) < \infty$;
\item [\rm (ii)] $T_f$ is Devaney chaotic;
\item [\rm (iii)] $T_f$ is frequently hypercyclic;
\item [\rm (iv)] $T_f$ has the OSP;
\item [\rm (v)] $T_f$ has the OSSP;
\item [\rm (vi)] $T_f$ has a nontrivial periodic point; 
\item [\rm (vii)] $T_f$ has a nontrivial fixed point;
\item [\rm (viii)] for each $q \in \N$, $T_f$ has a nontrivial periodic point with period $q$.
\end{itemize}
\end{thm}

\begin{proof}
Let $B_w$ be the bilateral weighted backward shift on $\ell^p(\Z)$ associated to the dissipative composition operator $T_f$.
Since
\[
w_{k} = \left(\frac{\mu(f^{k-1}(W))}{\mu(f^{k}(W))}\right)^{\frac{1}{p}} \ \ \text{ for all } k \in \Z,
\]
condition (i) in Corollary~\ref{CorBwlp} can be rewritten as
\[
\sum_{n=1}^{\infty} \frac{\mu(f^n(W))}{\mu(W)} < \infty \ \text{ and } \ \sum_{n=1}^{\infty} \frac{\mu(f^{-n}(W))}{\mu(W)}  < \infty,
\]
which is equivalent to condition (i) above.
Moreover, by \cite[Theorem M]{DAnielloDarjiMaiuriello2} and Theorem~\ref{thmshiftlike},
$T_f$ is Devaney chaotic (resp.\ is frequently hypercyclic, has the OSP, has the OSSP) if and only if so does $B_w$.
Therefore, the equivalences between (i)--(v) follow from the corresponding equivalences in Corollary~\ref{CorBwlp}.

Let us prove that (vi) $\Rightarrow$ (v).
Since both $f$ and $f^{-1}$ satisfy condition ($\star$), there exists a constant $c > 0$ such that
\[
\mu(f^{-1}(B)) \leq c\, \mu(B) \ \ \text{ and } \ \ \mu(f(B)) \leq c\, \mu(B) \ \text{ for all } B \in {\mathcal B}.
\]
Consequently, a simple application of the monotone convergence theorem shows that
\begin{equation}\label{S5-Eq1}
\int_{f^k(B)} \psi\, d\mu \leq c^{|k|} \int_B \psi \circ f^k d\mu,
\end{equation}
whenever $\psi : X \to [0,\infty]$ is measurable, $B \in {\mathcal B}$ and $k \in \Z$.
On the other hand, $B_w$ is a factor of $T_f$ by means of the linear quasi-conjugation $\Gamma : L^p(X) \to \ell^p(\Z)$ given by
\begin{equation}\label{S5-Eq2}
\Gamma(\varphi)(k) = \frac{\mu(f^k(W))^\frac{1}{p}}{\mu(W)} \int_W \varphi \circ f^k d\mu
\end{equation}
(see \cite[Lemma~4.2.3]{DAnielloDarjiMaiuriello} or \cite[Page~2]{DAnielloDarjiMaiuriello2}).
By hypothesis, $T_f$ has a nontrivial periodic point $\varphi \in L^p(X)$, say $(T_f)^q(\varphi) = \varphi$.
Replacing $\varphi$ with $|\varphi|$, if necessary, we can assume that $\varphi \geq 0$.
If $\Gamma(\varphi)$ were $0$, then (\ref{S5-Eq1}) and (\ref{S5-Eq2}) would imply that
\[
\int_{f^k(W)} \varphi\, d\mu \leq c^{|k|} \int_W \varphi \circ f^k d\mu = 0 \ \ \text{ for all } k \in \Z,
\]
which would give $\varphi = 0$, a contradiction.
Thus, $x = \Gamma(\varphi)$ is a nonzero vector in $\ell^p(\Z)$.
Since $\Gamma \circ T_f = B_w \circ \Gamma$, we obtain
\[
(B_w)^q(x) = (B_w)^q(\Gamma(\varphi)) = \Gamma((T_f)^q(\varphi)) = \Gamma(\varphi) = x.
\]
This proves that $B_w$ has a nontrivial periodic point.
By Corollary~\ref{CorBwlp}, $B_w$ has the OSSP and, therefore, so does $T_f$ in view of Theorem~\ref{thmshiftlike}(b).

Finally, since the implications (viii) $\Rightarrow$ (vii) $\Rightarrow$ (vi) are trivial, it remains to prove that (i) $\Rightarrow$ (viii).
For this purpose, take $q \in \N$. Since $\mu(X) < \infty$, the series 
\[
\varphi = \sum_{n = - \infty}^{\infty} \rchi_{f^{nq}(W)}
\]
converges and defines an element of $L^p(X)$. Since
\[
(T_f)^j(\varphi) = \varphi \circ f^j = \sum_{n = - \infty}^{\infty} \rchi_{f^{nq}(W)} \circ f^j = \sum_{n = - \infty}^{\infty} \rchi_{f^{nq-j}(W)}
\]
is equal to $\varphi$ for $j = q$ and is different from $\varphi$ for $j \in \{1,\ldots,q-1\}$,
we have that $\varphi$ is a nontrivial periodic point of $T_f$ with period $q$.
\end{proof}

\begin{rmk}
In \cite[Observation 4]{Bartoll2016} the authors point out that, 
although they removed the compactness assumption of each $K_m$ from the definition of the OSP, 
it is hard to think of a map having the specification property outside of the compact setting; 
in other words, for all cases they know of operators having the OSP, the required sets $K_m$ are always compact.
As we have just shown in Theorem \ref{thmMAINT_f}, the OSP and the OSSP are equivalent
for dissipative composition operators of bounded distortion. 
Hence, the question remains: 
{\em it is evident that the OSSP implies the more general OSP, but are there any operators that have the OSP but not the OSSP?}
\end{rmk}

\subsection{A special case: when $T_f= T_{\alpha}$, the translation operator}

Let $(X,\B,\mu,f)$ be a dissipative composition dynamical system generated by $W$. 
Let $\frac{d \mu(f^k)}{d \mu}$ denote the Radon-Nikodym derivative of $\mu(f^k)$ with respect to $\mu$. 
We recall the following definition from \cite{DAnielloDarjiMaiuriello}. 

\begin{defn} 
Take $\alpha > 0$. As in \cite[Section 3.2]{DAnielloDarjiMaiuriello} with $\alpha =1$, 
take $X = \R$, $\B$ the collection of Borel subsets of $\R$, and $f_{\alpha}(x) = x + \alpha$. 
Note that, independent of the measure $\mu$ we choose on $\R$, we get a dissipative system generated by $W = [0,\alpha)$. 
Let $\mu$ on $\B$ be generated by a density $h$, i.e.,
\[
\mu(B) = \int_{B} h\, dx, 
\]
where $h$ is a non-negative locally integrable Lebesgue measurable function. 
Then $(\R,\B,\mu)$ is a $\sigma$-finite measure space. 
Let $T_{\alpha} = T_{f_\alpha}$ be the {\em translation operator} defined on $L^{p}(\R,\B,\mu)$, $1 \leq p < \infty$, as 
\[
T_{\alpha} \varphi (x) = (\varphi \circ f_{\alpha})(x)=  \varphi (x+ \alpha)
\] 
for every $x \in {\mathbb R}$. 
\end{defn}

As it is simple to compute, for every $k \in \Z$ and every $x \in W$ with $h(x) \neq 0$,  
\[ 
\frac{d \mu(f_{\alpha}^k)}{d \mu} = \frac{h(x + k \alpha)}{h(x)}\cdot
\]

The next result follows directly from \cite[Theorem 2.1.1]{SinghManhas1993}.

\begin{prop}
$T_{\alpha}$ is an invertible continuous composition operator on $L^p(\R,\B,\mu)$, that is, 
it maps, together with its inverse, $L^p(\R,\B,\mu)$ continuously into itself, if and only if 
there exists a constant $C > 0$ such that
\[
\max\{h(x+\alpha), h(x-\alpha)\} \leq C\, h(x) \ \ \text{ for all } x \in \R.  \hspace{2cm} (\star \star)
\]
\end{prop}

The next result follows directly from Theorem \ref{thmMAINT_f}.

\begin{cor}[Equivalences for $T_{\alpha}$] \label{thmMAINTa}
Consider $(\R,\B,\mu,f_{\alpha})$ the dissipative composition dynamical system, generated by $W = [0, \alpha)$.
Let $T_{\alpha} : L^p(\R,\B,\mu) \to L^p(\R,\B,\mu)$ be of bounded distortion, with $\mu$ given by a density $h$.  
Then, the following conditions are equivalent:
\begin{itemize}
\item [\rm (i)] $\mu(\R) < \infty$ (equivalently, $h \in L^{1}(\R)$); 
\item [\rm (ii)] $T_{\alpha}$ is Devaney chaotic;
\item [\rm (iii)] $T_{\alpha}$ is frequently hypercyclic;
\item [\rm (iv)] $T_{\alpha}$ has the OSP;
\item [\rm (v)] $T_{\alpha}$ has the OSSP;
\item [\rm (vi)] $T_{\alpha}$ has a nontrivial periodic point; 
\item [\rm (vii)] $T_{\alpha}$ has a nontrivial fixed point;
\item [\rm (viii)] for each $q \in \N$, $T_{\alpha}$ has a nontrivial periodic point with period $q$.
\end{itemize}
\end{cor}

We recall the following result from \cite[Proposition 2.6.6]{DAnielloDarjiMaiuriello}.

\begin{prop}[Bounded RN Condition] \label{prodDDM} 
Let $(X, {\mathcal B}, \mu, f)$ be a dissipative system generated by $W$. 
Let ${\rho}_k = \frac{d \mu(f^k)}{d \mu}$, $m_{k}=  {\essinf}_{x \in W}{{\rho}}_{k}(x)$ and $M_{k} = \esssup_{x \in W} {\rho}_{k}(x)$. 
If $\{\frac{M_{k}}{m_{k}}\}_{k \in {\mathbb Z}}$ is bounded, then $f$ is of bounded distortion on $W$.
\end{prop}

\begin{exmp} 
Take the standard Cauchy distribution on $\R$, defined by the following probability density function
\[
h(x) = \frac{1}{\pi (1 + x^2)}\cdot
\]
Let $\mu$ be the probability measure whose density is $h$. Then, for each $k \in \Z$,
\[
\frac{d \mu(f_{\alpha}^k)}{d \mu} = \frac{h(x + k \alpha)}{h(x)} = \frac{1 + x^2}{1 + {(x + k \alpha)}^2}\cdot
\]
For $x \in W = [0, \alpha)$ and $k \geq 0$, we have
\[
\frac{1}{1 + {[(k+1) \alpha]}^2}  \leq \frac{1 + x^2}{1 + {(x + k \alpha)}^2}  \leq  \frac{1+ {\alpha}^{2}}{1 + {(k \alpha)}^2}\,,
\]
which implies that 
\[
M_k \leq \frac{1+ {\alpha}^{2}}{1 + {(k \alpha)}^2} \ \ \text{ and } \ \ m_k \geq \frac{1}{1 + {[(k+1) \alpha]}^2} \ \text{ for every } k \geq 0.
\] 
For $x \in W = [0, \alpha)$ and $k \leq 0$, we have
\[
\frac{1}{1 + {(k \alpha)}^2}  \leq \frac{1 + x^2}{1 + {(x + k \alpha)}^2}  \leq  \frac{1+ {\alpha}^{2}}{1 + {[(k+1) \alpha]}^2}\,,
\]
which implies that 
\[
M_k \leq \frac{1+ {\alpha}^{2}}{1 + {[(k+1) \alpha]}^2} \ \ \textrm{ and } \ \ m_k \geq \frac{1}{1 + {(k \alpha)}^2} \ \text{ for every } k \leq 0.
\]
Therefore, $\rho_k$ is bounded and, hence, by the Bounded RN Condition in Proposition~\ref{prodDDM}, 
$f_{\alpha}$ is of bounded distortion. 
So, $(\R,\B,\mu,f_\alpha)$ is a dissipative composition dynamical system generated by $W = [0, \alpha)$, of bounded distortion, with
\[
\mu(\R) = \int_{\R} h\, dx = \int_{\R} \frac{1}{\pi (1+ x^2)}\, dx = 1.
\]
Thus, by Theorem \ref{thmMAINTa}, $T_{\alpha}$ has the OSSP.
\end{exmp}

\section{The important role played by the bounded distortion}

In the study of dissipative composition operators on $L^p(X)$ spaces, 
the bounded distortion hypothesis appears in the statement of some theorems. 
For instance, the following characterization of the shadowing property was obtained in \cite{DAnielloDarjiMaiuriello}.

\begin{thm}\label{ShadTf}
Fix a real number $p \in [1,\infty)$ and let $X$, $\B$, $\mu$ and $f$ be as in Section~\ref{s-OSP,Tf}.
Suppose that the composition operator $T_f$ on $L^p(X)$ is dissipative with wandering set $W$.
If $T_f$ has bounded distortion, then $T_f$ has the shadowing property if and only if one of the following conditions holds: 
\begin{equation*}
\limsup_{n \to \infty} \sup_{k \in \Z} \left(\frac{\mu(f^{k}(W))}{\mu(f^{k+n}(W))}\right)^{\frac{1}{n}} < 1, \tag*{($\mathcal{HC}$)}   
\end{equation*}
\begin{equation*}
\liminf_{n \to \infty} \inf_{k \in \Z} \left(\frac{\mu(f^{k}(W))}{\mu(f^{k+n}(W))}\right)^{\frac{1}{n}} > 1, \tag*{($\mathcal{HD}$)} 
\end{equation*}
\begin{equation*}
\limsup_{n \to \infty} \sup_{k \in -\N_0} \left(\frac{\mu(f^{k-n}(W))}{\mu(f^{k}(W))}\right)^{\frac{1}{n}} < 1 
  \ \  \& \ \ 
\liminf_{n \to \infty} \inf_{k \in \N_0} \left (\frac{\mu(f^{k}(W))}{\mu(f^{k+n}(W))}\right)^{\frac{1}{n}} > 1. \tag*{($\mathcal{GH}$)}
\end{equation*}
\end{thm}

\medskip
Another example where the bounded distortion hypothesis is assumed can be found in the characterization 
of frequent hypercyclicity and Devaney chaos for dissipative composition operators obtained in \cite{Darjipires}.

However, to the best of the authors' knowledge, whether the bounded distortion hypothesis 
is actually essential for the validity of such theorems remained an open question.

In this section, we begin to develop an investigation into the importance of the bounded distortion hypothesis,
presenting the example below, which shows that it is indeed essential for the validity of Theorem~\ref{ShadTf}.

\begin{exmp} \label{ex1}
Let 
\[
X = \{a_j : j \in \Z\} \cup \{b_j : j \in \Z\},
\]
where the $a_j$'s and $b_j$'s are pairwise distinct objects.
Let $\B = {\cal P}(X)$ (the power set of $X$) and let $f : X \to X$ be given by 
\[
f(a_j) = a_{j+1} \ \ \text{ and } \ \ f(b_j) = b_{j+1}.
\]
For each $j \in \Z$, let $W_j = \{a_j,b_j\}$, and define $W = W_0$. Clearly,
\[
X = \dot\bigcup_{k \in \Z}\, f^k(W) \ \ \ \ \ (\text{disjoint union}),
\]
so that we have dissipativity.
We define a measure $\mu$ on $\B$ by putting \\
\begin{center}
$\begin{cases}
   \mu(\{a_j\}) = \mu(\{b_j\}) = \frac{1}{2^{-j+1}}   & \text{for }j \leq 0 \\
    \mu(\{a_j\}) = 2^j - 1 \text{ and }  \mu(\{b_j\})=1  & \text{for } j \geq 1.
\end{cases}$
\end{center}
Note that
\[
\mu(W_j) = 2^j \ \ \text{ for all } j \in \Z.
\]
Since $f^i(W) = W_i$ for each $i \in \Z$, we have that, for each $j \in \Z$ and each $n \in \Z$, 
\[
\frac{\mu(f^j(W))}{\mu(f^{j+n}(W))} = \frac{\mu(W_j)}{\mu(W_{j+n})} = \frac{2^j}{2^{j+n}} = \frac{1}{2^n}\cdot
\]
Therefore,
\[
\limsup_{n \to \infty} \sup_{k \in \Z} \left(\frac{\mu(f^{k}(W))}{\mu(f^{k+n}(W))}\right)^{\frac{1}{n}} = \frac{1}{2} < 1,
\] 
showing that condition $(\cal HC)$ of Theorem~\ref{ShadTf} is satisfied.
However, we shall prove that the composition operator $T_f$ on $L^p(X)$ does not have the shadowing property. 
Indeed, suppose that $T_f$ has the shadowing property and let $\delta > 0$ be associated to $\epsilon = 1$ according to this property.
We consider the $\delta$-pseudotrajectory $(\varphi_j)_{j \in \Z}$ of $T_f$ given by
\begin{align*}
\varphi_j &= 0 \ \ \text{ for } j \geq 0;\\
\varphi_j &= T_f^{-1}(\varphi_{j+1}) + \delta\, {\rchi_{\{b_{-j}\}}} \ \ \text{ for } j \leq -1.
\end{align*}
Note that, for each $j \geq 1$, $\varphi_{-j}$ is always 0 except on $b_j$, where $\varphi_{-j}(b_j) = j \delta$. 
By hypothesis, there exists $\psi \in L^p(X)$ such that
\[
\|\varphi_j - (T_f)^j(\psi)\|_p < 1 \ \ \text{ for all } j \in \Z.
\]
In particular, for each $j \in \N$,  
\[ 
|\varphi_{-j}(b_j) - \psi(f^{-j}(b_j))|^p \mu(\{b_j\}) \leq \|\varphi_{-j} - (T_f)^{-j}(\psi)\|_p^p < 1.
\]
This implies that  
\[ 
{|j \delta - \psi(b_0)|}^{p} < 1 \ \ \text{ for all } j \in \N,
\]
which is impossible. 
Therefore, $T_f$ does not have the shadowing property, although condition ($\cal HC$) is satisfied.
\end{exmp}

Let us now show that the bounded distortion hypothesis is also essential for the validity of Theorem~\ref{thmMAINT_f}.

\begin{exmp}\label{BDH-Necessary}
Let $X$, $\B$, $f$, $W_j$ ($j \in \Z$) and $W$ be as in the previous example.
The measure $\mu$ on $\B$ is now defined by setting
\[
\mu(\{a_j\}) = 2^{-|j|} \ \ \text{ and } \ \ \mu(\{b_j\}) = 1 \ \ \text{ for all } j \in \Z.
\]
For each $q \in \N$, let
\[
A_q = \{a_{jq} : j \in \Z\}.
\]
Clearly, the function $\rchi_{A_q}$ belongs to $L^p(X)$ and is a nontrivial periodic point of $T_f$ with period $q$.
Thus, $T_f$ has nontrivial periodic points of every period, that is, property (viii) in Theorem~\ref{thmMAINT_f} holds.
In particular, properties (vii) and (vi) also hold.
However, we claim that all the other properties, namely (i)--(v), fail.
Indeed, if $\varphi \in L^p(X)$ is a periodic point of $T_f$ with a certain period $q \in \N$, then
\[
\varphi(b_j) = ((T_f)^q(\varphi))(b_j) = \varphi(b_{j+q}) \ \ \text{ for all } j \in \Z,
\]
showing that the sequence $(\varphi(b_j))_{j \in \Z}$ is periodic.
Since the function $|\varphi|^p$ is integrable on $X$ and $\mu(\{b_j\}) = 1$ for all $j \in \Z$, we conclude that
\[
\varphi(b_j) = 0 \ \ \text{ for all } j \in \Z.
\]
In particular, this implies that the set of all periodic points of $T_f$ is not dense in $L^p(X)$.
Hence, properties (ii), (iv) and (v) are false.
It is obvious that property (i) is false.
Finally, property (iii) is also false, because the operator $T_f$ is not even hypercyclic.
This follows from the fact that the restriction of $T_f$ to the closed $T_f$-invariant subspace
\[
N = \{\varphi \in L^p(X) : \varphi(a_j) = 0 \text{ for all } j \in \Z\}
\]
is an isometry.
\end{exmp}

Before presenting the next example, we would like to recall some connections between the various versions 
of the Frequent Hypercyclicity Criterion (FHC) and the properties analysed in this paper. 
They hold in general settings and, as we shall see later, they play an important role also in the specific context 
of dissipative composition operators.
So, let us recall the general version of the FHC provided by Bonilla and Grosse-Erdmann,
which is contained in \cite[Theorem 2.4]{BONILLA_GROSSE-ERDMANN_2007}.

\begin{thm}[{\bf{Frequent Hypercyclicity Criterion 1 (FHC1)}}] \label{FH1}
Let $T$ be an operator on a separable $F$-space $X$. 
If there exist a dense subset $X_0$ of $X$ and a sequence of maps $S_n : X_0 \to X$ such that, for each $x \in X_0$,
\begin{enumerate}
\item[\rm (i)]{$\sum_{n=1}^{\infty}T^n(x)$ converges unconditionally,}
\item[\rm (ii)]{$\sum_{n=1}^{\infty} S_n(x)$ converges unconditionally,}
\item[\rm (iii)]{$T^n S_n (x) = x$ and $T^m S_n(x) = S_{n-m}(x)$ if $ n > m$,}
\end{enumerate}
then the operator T is frequently hypercyclic.
\end{thm}

The FHC1 is far stronger than any of the dynamical properties we have mentioned in this paper and, in particular, 
it implies the OSP  (in fact, the OSSP, as the invariant $K_m$ in the proof are all compact) as showed in \cite[Theorem 14]{Bartoll2016}.
The following version of the FHC is the most popular (see \cite[Theorem~9.9]{GrosseErdmannPeris2011}, for instance).

\begin{thm}[{\bf{Frequent Hypercyclicity Criterion 2 (FHC2)}}] \label{FH2} 
Let $T$ be an operator on a separable F-space $X$. 
If there exist a dense subset $X_0$ of $X$ and a map $S : X_0 \to X_0$ such that, for any $x \in X_0$,
\begin{enumerate}
\item[\rm (a)] {$\sum_{n=1}^{\infty} T^n(x)$ converges unconditionally,}
\item[\rm (b)] {$\sum_{n=1}^{\infty} S^n(x)$ converges unconditionally,}
\item[\rm (c)] {$TS(x) = x$,}
\end{enumerate}
then the operator $T$ is frequently hypercyclic.   
\end{thm}

The FHC2 clearly follows from the FHC1 by taking $S_n = S$ for all $n \geq 1$. 
It is a strengthened version of the original criterion obtained by Bayart and Grivaux in \cite[Theorem 2.1]{BayartGrivaux2006}. 
It is well known that the FHC2 also implies Devaney chaos and topological mixing.
Moreover, it is linked to the following condition:

\begin{defn} (Summability Condition)  
A measurable system $(X,\B,\mu,f)$ (or simply the map $f$) is said to satisfy the {\em Summability Condition (SC)} if, 
for each $\eps > 0$ and each $B \in \B$ with $\mu(B) < \infty$, there exists a measurable set $B' \subseteq B$ such that
\[
\mu(B \backslash B') < \eps \ \text{ and } \ \sum_{n \in \Z} \mu(f^n(B')) < \infty. \tag{SC}
\]
\end{defn}

The next theorem, provided by Darji and Pires in \cite[Theorem 3.1]{Darjipires}, 
shows that the SC is the natural translation of the FHC2 to the composition operator framework.

\begin{thm} \label{DP}
Let $(X,\B,\mu,f)$ be a measurable system. 
For all $p \geq 1$, SC implies FHC2. Moreover, SC and FHC2 are equivalent for all $p \geq  2$.
\end{thm}

They also established the following result \cite[Theorem 3.3]{Darjipires}.

\begin{thm}\label{DP2}
Let $(X,\B,\mu,f)$ be a measurable system.
If $\mu(X)$ is finite, then $f$ satisfies the SC if and only if $f$ is dissipative.
\end{thm}

Now, back to composition operators, we have the following remark.

\begin{rmk}\label{rmkOSSP}
Let $(X,\B,\mu,f)$ be an invertible dissipative measurable system.
Suppose that $\mu(X)$ is finite.
By Theorem~\ref{DP2}, $f$ satisfies the SC.
Hence, by Theorem~\ref{DP}, $T_f$ satisfies the FHC2.
Therefore, properties (ii), (iii), (iv), (v) and (vi) in Theorem~\ref{thmMAINT_f} hold.  
Property (viii) (hence, property (vii)) in this theorem also hold, 
since the proof of (i) $\Rightarrow$ (viii) in this theorem does not use the bounded distortion hypothesis.
The conclusion is that:
{\em Property (i) implies all the other properties in Theorem~\ref{thmMAINT_f} even if we remove the bounded distortion hypothesis.}
On the other hand, we saw in Example~\ref{BDH-Necessary} that, without the bounded distortion hypothesis, 
it is false in general that properties (vi), (vii) and (viii) imply property (i).
In the next example, we will see that properties (ii), (iii), (iv) and (v) do not imply property (i), in general, if we remove the bounded distortion hypothesis.
\end{rmk}

\begin{exmp} \label{infty}
We will now obtain an invertible dissipative measurable system $(X,\B,\mu,f)$, with a purely atomic measure $\mu$, 
such that $\mu(X) = \infty$ and $f$ satisfies the SC.
As a consequence, the invertible dissipative composition operator $T_f$ on $L^p(X)$ satisfies the FHC2 and, therefore, it has the OSSP. 
In particular, $T_f$ is frequently hypercyclic, Devaney chaotic and topologically mixing. 
 
\smallskip
To this aim, let $X = \Z \times \Z$ and $\B = {\cal P}(X)$. 
The measure $\mu$ on $\B$ is defined by putting
\[
\mu(\{(i,j)\}) = \frac{1}{2^{|i+j|}}\cdot
\] 
The transformation $f : X \to X$ is given by 
\[
f(i,j) = (i,j+1).
\] 
Since
\[
\mu(f^{-1}(\{(i,j)\})) = \frac{1}{2^{|i+j-1|}} \ \ \text{ and } \ \ \mu(f(\{(i,j)\})) = \frac{1}{2^{|i+j+1|}}\,,
\]
we have that
\[
\frac{\mu(f^{-1}(\{(i,j)\}))}{\mu(\{(i,j)\})} = \frac{2^{|i+j|}}{2^{|i+j-1|}} \leq 2 \ \ \text{ and } \ \
\frac{\mu(f(\{(i,j)\}))}{\mu(\{(i,j)\})} = \frac{2^{|i+j|}}{2^{|i+j+1|}} \leq 2.
\]
This shows that both $f$ and $f^{-1}$ satisfy condition ($\star$) with $c = 2$.
Hence, $(X,\B,\mu,f)$ is an invertible measurable system, which is dissipative, since
\[
X = \dot\bigcup_{k \in \Z} W_k = \dot\bigcup_{k \in \Z} f^k(W),
\] 
where $W_k = \Z \times \{k\}$, for each $k \in \Z$, and $W = W_0$. Since
\[
\mu(W_k) = \sum_{i \in \Z} \mu(\{(i,k)\}) = \sum_{i \in \Z} \frac{1}{2^{|i+k|}} = 3 \ \ \text{ for all } k \in \Z,
\]
we have that $\mu(X) = \infty$. It remains to show that $f$ satisfies the SC.
For this purpose, take $\eps > 0$ and $B \in \B$ with $\mu(B) < \infty$. Since
\[
\mu(B) = \sum_{(i,j) \in B} \mu(\{(i,j)\}) < \infty,
\]
there is a finite subset $B'$ of $B$ such that 
\[
\mu(B \backslash B') = \sum_{(i,j) \in B \backslash B'} \mu(\{(i,j)\}) < \eps.
\]
Moreover,
\begin{align*}
\sum_{n \in \Z} \mu(f^n(B')) &= \sum_{n \in \Z} \sum_{(i,j) \in B'} \mu(f^n(\{(i,j)\}))\\ 
  &= \sum_{(i,j) \in B'} \sum_{n \in \Z} \frac{1}{2^{|i+j+n|}} = 3\, \card(B') < \infty,
\end{align*}
where $\card(B')$ denotes the cardinality of the set $B'$.
Thus, $f$ satisfies the SC.
\end{exmp}

The next is an example of a dissipative system, without bounded distortion, that satisfies the OSSP. 

\begin{exmp}
We will now obtain an invertible dissipative measurable system $(X,\B,\mu,f)$ such that 
$\mu(X) < \infty$ but the bounded distortion does not hold.
In view of Remark~\ref{rmkOSSP}, the invertible dissipative composition operator $T_f$ on $L^p(X)$ has the OSSP,
although the bounded distortion fails.

\smallskip
Let $X$ be the product space $[0,1) \times \Z$, where $[0,1)$ is endowed with its usual topology
and $\Z$ is endowed with its discrete topology, and let $\B$ be the Borel $\sigma$-algebra on $X$.
Define  
\[ 
W_k = [0, 1) \times \{k\}, \ \ k \in \Z.
\]  
The transformation $f : X \to X$ is given by 
\[
f(x,k)= (x, k+1).
\]
Note that $f$ is a homeomorphism and, for each $ k \in \Z$, 
\[
f^k(W_0) = [0,1) \times \{0+k\} = W_k.
\] 
Hence, taking $W = W_0$, we have that 
\[
X = \dot\bigcup_{k \in \Z} W_k = \dot\bigcup_{k \in \Z} f^k(W),
\]
so that we have dissipativity.
The measure $\mu$ on $\B$ is defined in the following way: For each Borel set $E \subseteq X$, let
\[
\mu(E) = \sum_{k \in \Z} \int_{E \cap W_k} \rho(x,k)\, dx,
\]
where the density $\rho$ is given, for each $x \in [0, 1)$, by  
\[
\rho(x,k) = 
\begin{cases} 
\frac{1}{2^k} [(k+1)x +1] & \text{if } k \geq 0 \\ 
\frac{1}{2^{\vert k \vert}} & \text{if } k <0. 
\end{cases}
\]
Hence, for $k \geq 0$, 
\[
\mu(W_k) =  \frac{1}{2^k} \int_{0}^{1} [(k+1)x +1] dx = \frac{1}{2^k} \Big(\frac{k+1}{2} + 1\Big) = \frac{k+3}{2^{k+1}}\,, 
\]
and, for $k < 0$,
\[
\mu(W_k) = \frac{1}{2^{\vert k \vert}} \int_{0}^{1} 1\, dx = \frac{1}{2^{\vert k \vert}}\cdot
\]
The measure of $X$ is finite as
\begin{align*}
\mu(X) &=  \sum_{k \in \Z} \mu(W_k) = \sum _{k=-\infty}^{-1} \frac{1}{ 2^{\vert k \vert  }} + \sum _{k=0}^{\infty} \frac{k+3}{2^{k+1}}\\
&= 1 + 4 \sum_{k=0}^{\infty} \frac{k+3}{2^{k +3}} = 1 + 4 \sum_{s=3}^{\infty} \frac{s}{2^{s}} = 1 + 4 \cdot (1) = 5 <\infty.
\end{align*}
Note also that $f$ is a non-singular transformation, since 
\[
f^{-1}(E) \cap W_k =\{(x,k) : (x, k+1) \in E\},
\] 
meaning that $f^{-1}(E) \cap W_k$ is a copy of $E \cap W_{k+1}$, implying
\[
\mu(E) = 0 \ \Longleftrightarrow \ \mu(f^{-1}(E))=0.
\] 
Moreover, since the Radon-Nikodym derivative $h(x, k) = \frac{\rho(x, k-1)}{\rho(x, k)}$ is the ratio of the densities 
between two successive copies, an easy computation shows that, for each $(x,k) \in X$,   
\[
\frac{1}{4} \leq \frac{\rho(x, k-1)}{\rho(x, k)} \leq 4.
\]
This implies that both $f$ and $f^{-1}$ satisfy condition ($\star$).
Hence, $(X,\B,\mu,f)$ is an invertible dissipative measurable system.
It remains to show that the bounded distortion does not hold. 
For the bounded distortion to be satisfied, there must exist a constant $C > 0$ such that, 
for each measurable subset $B$ of $W = W_0$ and each $k  \in \Z$,  
\[
\frac{1}{C}\, \mu(f^{k}(W_{0})) \mu(B) \leq \mu(f^{k}(B)) \mu(W_{0}) \leq C\, \mu (f^{k}(W_{0})) \mu(B).
\]
We show that this inequality cannot be satisfied for $k \ge 0$ by choosing the subset $B$ appropriately. 
Let us take the subset $B_\eps = [0,\eps) \times \{0\}$ with an arbitrarily small $0 < \eps <1$, so that $B_\eps \subset W$, and note that
\[
\mu(W_0) = \frac{0+3}{2^1} = \frac{3}{2}
\]
and 
\[
\mu(B_{\eps}) = \int_0^{\eps} (x + 1)\, dx = \frac{\eps^2}{2} + \epsilon.
\] 
Moreover,
\[
\mu(f^k(W_0)) = \mu(W_k) = \frac{k+3}{2^{k+1}}
\]
and
\[
\mu (f^k(B_{\eps})) = \frac{1}{2^k} \int_{0}^{\eps} \left[(k+1)x+1\right]dx = \frac{1}{2^k} \left(\frac{k+1}{2}\,\eps^{2}+\eps\right).
\]
Substituting these values into the ratios, yields:  
\[
\frac{\mu(f^k(W_0))}{\mu(W_0)} \cdot \frac{\mu(B_{\eps})}{\mu(f^{k}(B_{\eps}))} 
= \frac{\frac{k+3}{2^{k+1}}}{\frac{3}{2}} \cdot \frac{\frac{\eps^2}{2}+\eps}{\frac{1}{2^k}\left(\frac{k+1}{2}\,\eps^2 + \eps\right)}
= \frac{k+3}{3} \cdot \frac{\frac{\eps}{2}+1}{\frac{k+1}{2}\,\eps +1}\cdot
\]
At this point, note that
\[
\lim_{\eps \to 0^+} \left(\frac{k+3}{3} \cdot \frac{\frac{\eps}{2}+1}{\frac{k+1}{2}\,\eps +1}\right)
= \frac{k+3}{3} \cdot \frac{1}{1} = \frac{k+3}{3}\cdot
\]
Since the choice of the constant $C$ in the definition of bounded distortion must hold uniformly for every $k \in \Z$, 
we observe that, by taking arbitrarily large values of $k$:
\[
\lim _{k \to \infty } \frac{k+3}{3} = \infty.
\]
Therefore, no finite constant $C$ exists that can upper bound this ratio and, then, the bounded distortion is not satisfied. 
\end{exmp}
 
As the next example shows, a non topologically mixing  map $f$ can generate a composition operator $T_{f}$ which is topologically mixing.  

\begin{exmp}
Recall that any irrational rotation of the circle is topologically transitive but not topologically mixing 
(see, for instance, \cite[Example 1.43]{GrosseErdmannPeris2011}).  
Consider $[0, 1)$ endowed with its usual topology. 
Fix any $\alpha \in [0, 1) \backslash \Q$ and define $g : [0, 1) \to [0, 1)$ as the irrational rotation:
\[
g(x) = x + \alpha \ (\text{mod } 1).
\]
Then, $g$ is invertible. The standard Lebesgue measure $m$ on $[0,1)$ is non-atomic, and $m([0, 1)) = 1$. 
Moreover, $g$ is an isometry, therefore it preserves the measure: 
\[
m(g(E)) = m(E),
\] 
for each measurable set $E \subseteq [0,1)$. In particular, $g$ is non-singular.
Let
\[ 
W = W_0 = [0, 1) \times \{0\} \ \text{ and } \ W_{k}= [0, 1) \times \{k\}, \ k \in \Z.
\]  
Define the space $X$ as the disjoint union: 
\[
X = \dot \bigcup_{k \in \Z} W_{k}.
\]
In other words, $X = [0,1) \times \Z$. 
We consider $X$ endowed with the product topology, where $\Z$ is given the discrete topology.
We now define the measure $\mu$ on $X$ by putting
\[
\mu(E \times \{k\}) = \frac{1}{2^{|k|+1}}\, m(E),
\]
for each measurable set $E \subseteq [0, 1)$ and each $k \in \Z$. 
Hence, the measure of $X$ is finite, since
\[
\mu(X) =  \frac{1}{2} + 2 \sum _{k=1}^{\infty} \frac{1}{2^{k+1}} = \frac{1}{2} + 1 =  \frac{3}{2}<\infty.
\]
Define $f : X \to X$ in the following way: 
\[
f(x,k)= (g(x), k+1) =  (x+\alpha \left( \text{mod } 1\right), k+1).
\]
Then, $f$ is invertible and non-singular. Moreover, $f$ is not topologically mixing as $g$ is not. 
Clearly, for each $ k \in \Z$, 
\[
f^{k}(W) = f^k(W_0) = [0,1) \times \{0+k\} =W_k,
\] 
meaning that $T_f$ is dissipative, since 
\[
X = \dot \bigcup_{k \in \Z} W_{k} = \dot \bigcup_{k \in \Z} f^{k}(W).
\]
Moreover, the bounded distortion property is satisfied: for any measurable subset $B \times \{0\}$ of $W$ and any $k \in \Z$,
\begin{align*}
\mu(f^k(B \times \{0\})) &= \mu(g^k(B) \times \{k\}) = \frac{1}{2^{\vert k \vert +1}}\, m(g^k(B))\\
&= \frac{1}{2^{\vert k \vert +1}}\, m(B) = \frac{1}{2^{\vert k \vert +1}}\, 2\, \mu(B \times \{0\}),
\end{align*}
and so
\[
\frac{ \mu(f^k(B \times \{0\}))} {\mu(B \times \{0\}) } = \frac{ 1 }{ 2^{\vert k \vert}}\cdot
\]
This number does not depend on the subset $B \times \{0\}$ of $W$ we take. 
It is the same number if we take the full $W$. So, we have bounded distortion.
Hence, it follows from Theorem~\ref{thmMAINT_f}  that $T_f$ satisfies the OSSP and, 
therefore, by \cite[Proposition~11]{Bartoll2016}, $T_{f}$ is topologically mixing.
\end{exmp}

\section*{Acknowledgement}

The first author was partially supported by CNPq (Brazil) -- Project {\#}308238/2021-4, by CAPES (Brazil) -- Finance Code 001,
and by Project PID2022-139449NB-I00, funded by MCIN/AEI/10.13039/501100011033/FEDER, UE. \\
The second author and the third author were partially supported by the “Gruppo Nazionale per l’Analisi Matematica, la Probabilità e le loro Applicazioni dell’Istituto Nazionale di Alta Matematica F. Severi”,  and 
their work was also partially accomplished within the UMI Group TAA “Approximation Theory and Applications”.\\
The second  author was also partially funded by “MEdical Device Maintenance with Artificial Intelligence - MEDMAI” (CUP B29J24000040005) from Ministero delle Imprese e del Made in Italy (MIMIT) "Ricerca collaborativa" of 
Programma Nazionale Ricerca, Innovazione e Competitività per la transizione verde e digitale 2021-2027.


\bibliographystyle{siam}
\bibliography{biblio}

@article {DAnielloMaiuriello2023,
    AUTHOR = {D'Aniello, Emma and Maiuriello, Martina},
    TITLE = {On the spectrum of weighted shifts},
    JOURNAL = {Rev. R. Acad. Cienc. Exactas F\'is. Nat. Ser. A Mat. RACSAM},
    FJOURNAL = {Revista de la Real Academia de Ciencias Exactas, F\'isicas y
              Naturales. Serie A. Matematicas. RACSAM},
    VOLUME = {117},
      YEAR = {2023},
    NUMBER = {1},
   MRNUMBER = {4493762},
   Note = {Paper No. 4, 19 pp.},
   ISSN = {1578-7303,1579-1505},
   MRCLASS = {47B37 (47A05 47A10 47A16)},
  
MRREVIEWER = {Tudor\ B\^inzar},
       DOI = {10.1007/s13398-022-01328-z},
       URL = {https://doi.org/10.1007/s13398-022-01328-z},
}

@article{Bartoll2014,
	Author = {Bartoll, S. and Mart\'inez-Gim\'enez, F. and Murillo-Arcila, M. and Peris, A.},
	Journal = {Descriptive Topology and Functional Analysis. Springer Proceedings in Mathematics and Statistics},
	Title = {Cantor Sets, {B}ernoulli Shifts and Linear Dynamics},
	Volume = {80},
	Year = {2014}}

@book{Conway,
	Author = {Conway, John},
	Publisher = {American Mathematical Society},
	Series = {Mathematical Surveys and Monographs},
	Title = {The Theory of Subnormal Operators},
	Volume = {36},
	Year = {1991}}

@article {BayartGrivaux2006,
    AUTHOR = {Bayart, Fr\'ed\'eric and Grivaux, Sophie},
     TITLE = {Frequently hypercyclic operators},
   JOURNAL = {Trans. Amer. Math. Soc.},
  FJOURNAL = {Transactions of the American Mathematical Society},
    VOLUME = {358},
      YEAR = {2006},
    NUMBER = {11},
     PAGES = {5083--5117},
      ISSN = {0002-9947,1088-6850},
   MRCLASS = {47A16 (37A05 37B05 47A35)},
  MRNUMBER = {2231886},
MRREVIEWER = {Vivien\ G.\ Miller},
       DOI = {10.1090/S0002-9947-06-04019-0},
       URL = {https://doi.org/10.1090/S0002-9947-06-04019-0},
}

@article{BayartRuzsa,
	Author = {Bayart, Fr\'{e}d\'{e}ric and Ruzsa, Imre Z.},
	Doi = {10.1017/etds.2013.77},
	Fjournal = {Ergodic Theory and Dynamical Systems},
	Issn = {0143-3857},
	Journal = {Ergodic Theory Dynam. Systems},
	Mrclass = {47A16 (37B50)},
	Mrnumber = {3334899},
	Mrreviewer = {Sophie Grivaux},
	Number = {3},
	Pages = {691--709},
	Title = {Difference sets and frequently hypercyclic weighted shifts},
	Url = {https://doi.org/10.1017/etds.2013.77},
	Volume = {35},
	Year = {2015}}

@inproceedings{Bowen1970,
  author    = {Rufus Bowen},
  title     = {\text{Topological Entropy and Axiom A}},
  booktitle = {Global Analysis},
  series    = {Proceedings of Symposia in Pure Mathematics},
  volume    = {14},
  address   = {Providence, RI},
  publisher = {American Mathematical Society},
  year      = {1970},
  pages     = {23--41},
  note      = {Proceedings of the Symposium in Pure Mathematics, Berkeley, CA, 1968}
}

@article {Bowen1971,
    AUTHOR = {Bowen, Rufus},
     TITLE = {Periodic points and measures for {A}xiom {$A$}
              diffeomorphisms},
   JOURNAL = {Trans. Amer. Math. Soc.},
  FJOURNAL = {Transactions of the American Mathematical Society},
    VOLUME = {154},
      YEAR = {1971},
     PAGES = {377--397},
      ISSN = {0002-9947,1088-6850},
   MRCLASS = {57.20},
  MRNUMBER = {282372},
MRREVIEWER = {U.\ D'Ambrosio},
       DOI = {10.2307/1995452},
       URL = {https://doi.org/10.2307/1995452},
}

@article {Bartoll2016,
    AUTHOR = {Bartoll, Salud and Mart\'inez-Gim\'enez, F\'elix and Peris,
              Alfredo},
     TITLE = {Operators with the specification property},
   JOURNAL = {J. Math. Anal. Appl.},
  FJOURNAL = {Journal of Mathematical Analysis and Applications},
    VOLUME = {436},
      YEAR = {2016},
    NUMBER = {1},
     PAGES = {478--488},
      ISSN = {0022-247X,1096-0813},
   MRCLASS = {47A16 (47B37)},
  MRNUMBER = {3440106},
MRREVIEWER = {Janko\ Bra\v ci\v c},
       DOI = {10.1016/j.jmaa.2015.12.004},
       URL = {https://doi.org/10.1016/j.jmaa.2015.12.004},
}

@article {Bartoll2012,
    AUTHOR = {Bartoll, Salud and Mart\'inez-Gim\'enez, F\'elix and Peris,
              Alfredo},
     TITLE = {The specification property for backward shifts},
   JOURNAL = {J. Difference Equ. Appl.},
  FJOURNAL = {Journal of Difference Equations and Applications},
    VOLUME = {18},
      YEAR = {2012},
    NUMBER = {4},
     PAGES = {599--605},
      ISSN = {1023-6198,1563-5120},
   MRCLASS = {47A16 (47B37)},
  MRNUMBER = {2905285},
       DOI = {10.1080/10236198.2011.586636},
       URL = {https://doi.org/10.1080/10236198.2011.586636},
}

@article {Brian2017,
    AUTHOR = {Brian, William R. and Kelly, James P. and Tennant, Tim},
     TITLE = {The specification property and infinite entropy for certain
              classes of linear operators},
   JOURNAL = {J. Math. Anal. Appl.},
  FJOURNAL = {Journal of Mathematical Analysis and Applications},
    VOLUME = {453},
      YEAR = {2017},
    NUMBER = {2},
     PAGES = {917--927},
      ISSN = {0022-247X,1096-0813},
   MRCLASS = {47A06},
  MRNUMBER = {3648265},
MRREVIEWER = {Koratti\ C.\ Sivakumar},
       DOI = {10.1016/j.jmaa.2017.04.041},
       URL = {https://doi.org/10.1016/j.jmaa.2017.04.041},
}

@article{DAnielloDarjiMaiuriello,
	Author = {D'Aniello, Emma and Darji, Udayan B. and Maiuriello, Martina},
	Doi = {10.1016/j.jde.2021.06.038},
	Fjournal = {Journal of Differential Equations},
	Issn = {0022-0396},
	Journal = {J. Differential Equations},
	Mrclass = {37D05 (37C50 47A16 47B33)},
	Mrnumber = {4284483},
	Pages = {68--94},
	Title = {Generalized hyperbolicity and shadowing in {$L^p$} spaces},
	Url = {https://doi.org/10.1016/j.jde.2021.06.038},
	Volume = {298},
	Year = {2021}}

@article{DAnielloDarjiMaiuriello2,
	Author = {D'Aniello, Emma and Darji, Udayan B. and Maiuriello, Martina},
	Doi = {10.1016/j.jmaa.2022.126393},
	Fjournal = {Journal of Mathematical Analysis and Applications},
	Issn = {0022-247X},
	Journal = {J. Math. Anal. Appl.},
	Mrclass = {47B33 (47A16)},
	Mrnumber = {4436502},
	Note = {Paper No. 126393, 13 pp.},
	Number = {1},
	Title = {Shift-like operators on {$L^p(X)$}},
	Url = {https://doi.org/10.1016/j.jmaa.2022.126393},
	Volume = {515},
	Year = {2022}}

@article{Buzzi,
	Author = {Buzzi, J\'er\^ome},
	Doi = {10.1090/S0002-9947-97-01873-4},
	Fjournal = {Transactions of the American Mathematical Society},
	Journal = {Trans. Amer. Math. Soc.},
	Number = {7},
	Pages = {2737--2754},
	Title = {Specification on the interval},
	Volume = {349},
	Year = {1997}}

@article{Darjipires,
	Author = {Darji, Udayan B. and Pires, Benito},
	Doi = {10.1017/S0013091521000286},
	Fjournal = {Proceedings of the Edinburgh Mathematical Society. Series II},
	Issn = {0013-0915},
	Journal = {Proc. Edinb. Math. Soc. (2)},
	Mrclass = {47A16 (37D45 47B33)},
	Mrnumber = {4330274},
	Number = {3},
	Pages = {513--531},
	Title = {Chaos and frequent hypercyclicity for composition operators},
	Url = {https://doi.org/10.1017/S0013091521000286},
	Volume = {64},
	Year = {2021}}

@article {Maiuriello2023,
    AUTHOR = {Maiuriello, Martina},
     TITLE = {On the existence of infinite-dimensional closed subspaces of
              frequently hypercyclic vectors for {$T_f$}},
   JOURNAL = {Real Anal. Exchange},
  FJOURNAL = {Real Analysis Exchange},
    VOLUME = {48},
      YEAR = {2023},
    NUMBER = {2},
     PAGES = {409--424},
      ISSN = {0147-1937,1930-1219},
   MRCLASS = {47A16 (47B33 91B55)},
  MRNUMBER = {4668957},
       DOI = {10.14321/realanalexch.48.2.1676962925},
       URL = {https://doi.org/10.14321/realanalexch.48.2.1676962925},
}

@article{Marcin2025,
	Author = {Deka, Konrad and Kwietniak, Dominik and Peng, Bo and Sabok, Marcin},
	Journal = {arXiv:2501.02723},
	Title = {Bowen's Problem 32 and the conjugacy problem for systems with specification},
	Year = {2025}}

@book{Ruette,
	Author = {Ruette, Sylvie},
	Doi = {10.1090/ulect/067},
	Isbn = {978-1-4704-2956-0},
	Mrclass = {37-02 (37D45 37E05)},
	Mrnumber = {3616574},
	Mrreviewer = {V\'ictor\ Jim\'enez L\'opez},
	Pages = {xii+215},
	Publisher = {American Mathematical Society, Providence, RI},
	Series = {University Lecture Series},
	Title = {Chaos on the {I}nterval},
	Url = {https://doi.org/10.1090/ulect/067},
	Volume = {67},
	Year = {2017}}

@article {Sigmund1974,
    AUTHOR = {Sigmund, Karl},
     TITLE = {On dynamical systems with the specification property},
   JOURNAL = {Trans. Amer. Math. Soc.},
  FJOURNAL = {Transactions of the American Mathematical Society},
    VOLUME = {190},
      YEAR = {1974},
     PAGES = {285--299},
      ISSN = {0002-9947,1088-6850},
   MRCLASS = {28A65 (54H20)},
  MRNUMBER = {352411},
MRREVIEWER = {J.\ Auslander},
       DOI = {10.2307/1996963},
       URL = {https://doi.org/10.2307/1996963},
}

@book {SinghManhas1993,
    AUTHOR = {Singh, R. K. and Manhas, J. S.},
     TITLE = {Composition {O}perators on {F}unction {S}paces},
    SERIES = {North-Holland Mathematics Studies},
    VOLUME = {179},
 PUBLISHER = {North-Holland Publishing Co., Amsterdam},
      YEAR = {1993},
     PAGES = {x+315},
      ISBN = {0-444-81593-7},
   MRCLASS = {47B38 (46E15 47-02)},
  MRNUMBER = {1246562},
MRREVIEWER = {Paul\ Bourdon},
}

@article {BCDMP18,
    AUTHOR = {Bernardes, Jr., Nilson C. and Cirilo, Patricia R. and Darji,
              Udayan B. and Messaoudi, Ali and Pujals, Enrique R.},
     TITLE = {Expansivity and shadowing in linear dynamics},
   JOURNAL = {J. Math. Anal. Appl.},
  FJOURNAL = {Journal of Mathematical Analysis and Applications},
    VOLUME = {461},
      YEAR = {2018},
    NUMBER = {1},
     PAGES = {796--816},
      ISSN = {0022-247X,1096-0813},
   MRCLASS = {37C50},
  MRNUMBER = {3759568},
MRREVIEWER = {Sophie\ Grivaux},
       DOI = {10.1016/j.jmaa.2017.11.059},
       URL = {https://doi.org/10.1016/j.jmaa.2017.11.059},
}

@article {BeMe21,
    AUTHOR = {Bernardes, Jr., Nilson C. and Messaoudi, Ali},
     TITLE = {Shadowing and structural stability for operators},
   JOURNAL = {Ergodic Theory Dynam. Systems},
  FJOURNAL = {Ergodic Theory and Dynamical Systems},
    VOLUME = {41},
      YEAR = {2021},
    NUMBER = {4},
     PAGES = {961--980},
      ISSN = {0143-3857,1469-4417},
   MRCLASS = {47A16 (37B99 37C20 37C50)},
  MRNUMBER = {4223416},
       DOI = {10.1017/etds.2019.107},
       URL = {https://doi.org/10.1017/etds.2019.107},
}

@article {BePe24,
    AUTHOR = {Bernardes, Jr., Nilson C. and Peris, Alfred},
     TITLE = {On shadowing and chain recurrence in linear dynamics},
   JOURNAL = {Adv. Math.},
  FJOURNAL = {Advances in Mathematics},
         Note = {Paper No. 109539, 46 pp.},
       VOLUME = {441},
        YEAR = {2024},
      ISSN = {0001-8708,1090-2082},
   MRCLASS = {37B65 (37B20 46A45 47A16 47B37)},
  MRNUMBER = {4708149},
       DOI = {10.1016/j.aim.2024.109539},
       URL = {https://doi.org/10.1016/j.aim.2024.109539},
}

@article {BCDFP25,
    AUTHOR = {Bernardes, Jr., Nilson C. and Caraballo, Blas M. and Darji,
              Udayan B. and F\'avaro, Vin\'icius V. and Peris, Alfred},
     TITLE = {Generalized hyperbolicity, stability and expansivity for
              operators on locally convex spaces},
   JOURNAL = {J. Funct. Anal.},
  FJOURNAL = {Journal of Functional Analysis},
    VOLUME = {288},
      YEAR = {2025},
        Note = {Paper No. 110696, 51 pp.},
      NUMBER = {2},
      ISSN = {0022-1236,1096-0783},
   MRCLASS = {47A16 (37B05 37B25 37B65 37C50 37D20)},
  MRNUMBER = {4799325},
MRREVIEWER = {Jan\ Hendrik\ Fourie},
       DOI = {10.1016/j.jfa.2024.110696},
       URL = {https://doi.org/10.1016/j.jfa.2024.110696},
}

@article {ABM21,
    AUTHOR = {Alves, Fabricio F. and Bernardes, Jr., Nilson C. and
              Messaoudi, Ali},
     TITLE = {Chain recurrence and average shadowing in dynamics},
   JOURNAL = {Monatsh. Math.},
  FJOURNAL = {Monatshefte f\"ur Mathematik},
    VOLUME = {196},
      YEAR = {2021},
    NUMBER = {4},
     PAGES = {665--697},
      ISSN = {0026-9255,1436-5081},
   MRCLASS = {37B20 (37B65 37C50 47A16)},
  MRNUMBER = {4331098},
       DOI = {10.1007/s00605-021-01617-6},
       URL = {https://doi.org/10.1007/s00605-021-01617-6},
}

@article {BONILLA_GROSSE-ERDMANN_2007,
    AUTHOR = {Bonilla, A. and Grosse-Erdmann, K.-G.},
     TITLE = {Frequently hypercyclic operators and vectors},
   JOURNAL = {Ergodic Theory Dynam. Systems},
  FJOURNAL = {Ergodic Theory and Dynamical Systems},
    VOLUME = {27},
      YEAR = {2007},
    NUMBER = {2},
     PAGES = {383--404},
      ISSN = {0143-3857,1469-4417},
   MRCLASS = {47A16 (37B05)},
  MRNUMBER = {2308137},
MRREVIEWER = {Sophie\ Grivaux},
       DOI = {10.1017/S014338570600085X},
       URL = {https://doi.org/10.1017/S014338570600085X},
}

@article {Omb94,
    AUTHOR = {Ombach, Jerzy},
     TITLE = {The shadowing lemma in the linear case},
   JOURNAL = {Univ. Iagel. Acta Math.},
  FJOURNAL = {Universitatis Iagellonicae. Acta Mathematica},
    NUMBER = {31},
      YEAR = {1994},
     PAGES = {69--74},
      ISSN = {0083-4386,2084-3828},
   MRCLASS = {58F15 (47A99)},
  MRNUMBER = {1290742},
MRREVIEWER = {Tadeusz\ Nadzieja},
}

@book {GrosseErdmannPeris2011,
    AUTHOR = {Grosse-Erdmann, Karl-G. and Peris Manguillot, Alfredo},
     TITLE = {Linear {C}haos},
    SERIES = {Universitext},
 PUBLISHER = {Springer, London},
      YEAR = {2011},
     PAGES = {xii+386},
      ISBN = {978-1-4471-2169-5},
   MRCLASS = {47-02 (37D45 47A16)},
  MRNUMBER = {2919812},
MRREVIEWER = {E.\ A.\ Gallardo-Guti\'errez},
       DOI = {10.1007/978-1-4471-2170-1},
       URL = {https://doi.org/10.1007/978-1-4471-2170-1},
}

@article{Pit,
	AUTHOR = {Pituk, Mih\'aly},
	TITLE = {Spectral characterization of shadowing for linear operators on {H}ilbert spaces},
	JOURNAL = {arXiv:2511.12272},
	YEAR = {2025}
}

@article {Grosse2000,
    AUTHOR = {Grosse-Erdmann, K.-G.},
     TITLE = {Hypercyclic and chaotic weighted shifts},
   JOURNAL = {Studia Math.},
  FJOURNAL = {Studia Mathematica},
    VOLUME = {139},
      YEAR = {2000},
    NUMBER = {1},
     PAGES = {47--68},
      ISSN = {0039-3223,1730-6337},
   MRCLASS = {47B37 (30E10 47A16 47B38)},
  MRNUMBER = {1763044},
MRREVIEWER = {Luis\ Bernal Gonz\'alez},
       DOI = {10.4064/sm-139-1-47-68},
       URL = {https://doi.org/10.4064/sm-139-1-47-68},
}

@article {LSV2026,
    AUTHOR = {Lupatini, Paulo and Carvalho Silva, Felipe and Var\~ao,
              R\'egis},
     TITLE = {Entropy for compact operators and results on entropy and
              specification},
   JOURNAL = {Monatsh. Math.},
  FJOURNAL = {Monatshefte f\"ur Mathematik},
    VOLUME = {209},
      YEAR = {2026},
    NUMBER = {3},
     PAGES = {497--511},
      ISSN = {0026-9255,1436-5081},
   MRCLASS = {47A16 (47A35 47B37)},
  MRNUMBER = {5033943},
       DOI = {10.1007/s00605-026-02158-6},
       URL = {https://doi.org/10.1007/s00605-026-02158-6},
}

@article {BBP2026,
    AUTHOR = {Bernardes, Jr., Nilson C. and Bonilla, Antonio and Pinto, Jo{\~a}o V. A.},
     TITLE = {On the dynamics of weighted composition operators},
   JOURNAL = {J. Math. Anal. Appl.},
  FJOURNAL = {Journal of Mathematical Analysis and Applications},
    VOLUME = {555},
      YEAR = {2026},
     Note = {Paper No. 130171, 45 pp.},
    NUMBER = {2},
      ISSN = {0022-247X,1096-0813},
   MRCLASS = {47B33 (47A16 47B37)},
  MRNUMBER = {4976736},
       DOI = {10.1016/j.jmaa.2025.130171},
       URL = {https://doi.org/10.1016/j.jmaa.2025.130171},
}

@article {GG2025,
    AUTHOR = {Gomes, Daniel and Grosse-Erdmann, Karl-G.},
     TITLE = {Kitai's criterion for composition operators},
   JOURNAL = {J. Math. Anal. Appl.},
  FJOURNAL = {Journal of Mathematical Analysis and Applications},
    VOLUME = {547},
      YEAR = {2025},
     Note = {Paper No. 129347, 28 pp.},
         NUMBER = {2},
      ISSN = {0022-247X,1096-0813},
   MRCLASS = {47B33 (37A25)},
  MRNUMBER = {4863798},
MRREVIEWER = {H\'ector\ Camilo\ Chaparro},
       DOI = {10.1016/j.jmaa.2025.129347},
       URL = {https://doi.org/10.1016/j.jmaa.2025.129347},
}

@article {AJK2025,
    AUTHOR = {Asensio, Vicente and Jord\'a, Enrique and Kalmes, Thomas},
     TITLE = {Power boundedness and related properties for weighted
              composition operators on {$\mathcal{S}(\Bbb{R}^d)$}},
   JOURNAL = {J. Funct. Anal.},
  FJOURNAL = {Journal of Functional Analysis},
    VOLUME = {288},
      YEAR = {2025},
     Note = {Paper No. 110745, 33 pp.},
         NUMBER = {3},
      ISSN = {0022-1236,1096-0783},
   MRCLASS = {47B33 (47A35)},
  MRNUMBER = {4823111},
MRREVIEWER = {Arkady\ K.\ Kitover},
       DOI = {10.1016/j.jfa.2024.110745},
       URL = {https://doi.org/10.1016/j.jfa.2024.110745},
}

\Addresses

\end{document}